\documentclass[11pt]{amsart}
\usepackage{amssymb}

\makeatletter
\newcommand{\sumprime}{\if@display\sideset{}{'}\sum%
            \else\sum'\fi}
\makeatother

\begin{document}

\numberwithin{equation}{section}

% define theorem environments
\newtheorem{theorem}{Theorem}[section]
\newtheorem{proposition}[theorem]{Proposition}
\newtheorem{conjecture}[theorem]{Conjecture}
\def\theconjecture{\unskip}
\newtheorem{corollary}[theorem]{Corollary}
\newtheorem{lemma}[theorem]{Lemma}
\newtheorem{observation}[theorem]{Observation}
\newtheorem{definition}{Definition}
\numberwithin{definition}{section} %\def\thedefinition{\unskip}
\newtheorem{remark}{Remark}
\def\theremark{\unskip}
\newtheorem{question}{Question}
\def\thequestion{\unskip}
\newtheorem{example}{Example}
\def\theexample{\unskip}
\newtheorem{problem}{Problem}

\def\vvv{\ensuremath{\mid\!\mid\!\mid}}
\def\intprod{\mathbin{\lr54}}
\def\reals{{\mathbb R}}
\def\integers{{\mathbb Z}}
\def\N{{\mathbb N}}
\def\complex{{\mathbb C}\/}
\def\dist{\operatorname{dist}\,}
\def\spec{\operatorname{spec}\,}
\def\interior{\operatorname{int}\,}
\def\trace{\operatorname{tr}\,}
\def\cl{\operatorname{cl}\,}
\def\essspec{\operatorname{esspec}\,}
\def\range{\operatorname{\mathcal R}\,}
\def\kernel{\operatorname{\mathcal N}\,}
\def\dom{\operatorname{Dom}\,}
\def\linearspan{\operatorname{span}\,}
\def\lip{\operatorname{Lip}\,}
\def\sgn{\operatorname{sgn}\,}
\def\Z{ {\mathbb Z} }
\def\e{\varepsilon}
\def\p{\partial}
\def\rp{{ ^{-1} }}
\def\Re{\operatorname{Re\,} }
\def\Im{\operatorname{Im\,} }
\def\dbarb{\bar\partial_b}
\def\eps{\varepsilon}
\def\O{\Omega}
\def\Lip{\operatorname{Lip\,}}

\def\Hs{{\mathcal H}}
\def\E{{\mathcal E}}
\def\scriptu{{\mathcal U}}
\def\scriptr{{\mathcal R}}
\def\scripta{{\mathcal A}}
\def\scriptc{{\mathcal C}}
\def\scriptd{{\mathcal D}}
\def\scripti{{\mathcal I}}
\def\scriptk{{\mathcal K}}
\def\scripth{{\mathcal H}}
\def\scriptm{{\mathcal M}}
\def\scriptn{{\mathcal N}}
\def\scripte{{\mathcal E}}
\def\scriptt{{\mathcal T}}
\def\scriptr{{\mathcal R}}
\def\scripts{{\mathcal S}}
\def\scriptb{{\mathcal B}}
\def\scriptf{{\mathcal F}}
\def\scriptg{{\mathcal G}}
\def\scriptl{{\mathcal L}}
\def\scripto{{\mathfrak o}}
\def\scriptv{{\mathcal V}}
\def\frakg{{\mathfrak g}}
\def\frakG{{\mathfrak G}}

\def\ov{\overline}

\thanks{Supported by Grant IDH1411041 from Fudan University}

\address{School of Mathematical Sciences, Fudan University, Shanghai 200433, China}
 \email{boychen@fudan.edu.cn}

\title{Bergman kernel and hyperconvexity index}
\author{Bo-Yong Chen}
\date{}
\maketitle

\bigskip

\centerline{{\small \it Dedicated to Professor John Erik Fornaess on the occasion of his 70-th birthday}}

\bigskip

\begin{abstract}
Let $\Omega\subset {\mathbb C}^n$ be a bounded domain with the hyperconvexity index $\alpha(\Omega)>0$. Let $\varrho$ be the relative extremal function of a fixed closed ball in $\Omega$ and set  $\mu:=|\varrho|(1+|\log|\varrho||)^{-1}$, $\nu:=|\varrho|(1+|\log|\varrho||)^n$.  We obtain the following estimates for the Bergman kernel: (1) For every $0<\alpha<\alpha(\Omega)$ and $2\le p<2+\frac{2\alpha(\Omega)}{2n-\alpha(\Omega)}$, there exists a constant $C>0$ such that
$\int_\Omega |\frac{K_\Omega(\cdot,w)}{\sqrt{K_\Omega(w)}}|^{p}\le C |\mu(w)|^{-\frac{(p-2) n}\alpha}$ for all $w\in \Omega$. (2) For every $0<r<1$, there exists a constant $C>0$ such that
$
 \frac{|K_\Omega(z,w)|^2}{K_\Omega(z)K_\Omega(w)}\le C (\min\{\frac{\nu(z)}{\mu(w)},\frac{\nu(w)}{\mu(z)}\})^r
$ for all $z,w\in \Omega$. Various application of these estimates are given.
\end{abstract}

\bigskip
\section{Introduction}

A domain $\Omega\subset {\mathbb C}^n$ is called\/ {\it hyperconvex\/} if there exists a negative continuous plurisubharmonic (psh) function $\rho$ on $\Omega$ such that $\{\rho<c\}\subset\subset \Omega$ for any $c<0$.
The class of hyperconvex domains is very wide, e.g. every bounded pseudoconvex domain with Lipschitz boundary is hyperconvex (cf. \cite{Demailly87}). Although hyperconvex domains already admit a rich function theory (see e.g. \cite{OhsawaHyperconvex}, \cite{BlockiPflug}, \cite{HerbortHyperconvex}, \cite{PoletskyHardy}), it is not enough to get quantitative results unless one imposes certain growth conditions on the bounded exhaustion function $\rho$ (compare \cite{BerndtssonCharpentier}, \cite{BlockiGreen}, \cite{DiederichOhsawa}).

A meaningful condition is $-\rho\le C \delta^\alpha$ for some constants $\alpha,C>0$, where $\delta$ denotes the boundary distance. Let $\alpha(\Omega)$ be the supremum of all $\alpha$.  We call it the\/ {\it hyperconvexity index\/} of $\Omega$.  From the fundamental work of Diederich-Fornaess \cite{DiederichFornaess77}, we know that if $\Omega$ is a bounded pseudoconvex domain with $C^2-$boundary then there exists a continuous negative psh function $\rho$ on $\Omega$ such that $C^{-1}\delta^\eta\le -\rho\le C \delta^\eta$ for some constants $\eta,C>0$. The supremum $\eta(\Omega)$ of all $\eta$ is called the\/ {\it Diederich-Fornaess index} of $\Omega$ (see e.g. \cite{Adachi}, \cite{FuShaw}, \cite{HarringtonPSH}). Clearly, one has $\alpha(\Omega)\ge \eta(\Omega)$. Recently,
 Harrington \cite{HarringtonPSH} showed that if $\Omega$ is a bounded pseudoconvex domain with Lipschitz boundary then $\eta(\Omega)>0$.

  On the other hand, there are plenty of domains with very irregular boundaries such that $\alpha(\Omega)>0$, while it is difficult to verify $\eta(\Omega)>0$. For instance, Koebe's distortion theorem implies $\alpha(\Omega)\ge 1/2$ if $\Omega\subsetneq {\mathbb C}$ is a simply-connected domain (see \cite{CarlesonGamelin}, Chapter 1, Theorem 4.4).  Recently, Carleson-Totik \cite{CarlesonTotik} and Totik \cite{Totik} obtained various Wiener-type criterions for planar domains with positive hyperconvexity indices. In particular, if $\partial \Omega$ is uniformly perfect in the sense of Pommerenke \cite{PommerenkePerfect}, then $\alpha(\Omega)>0$ (see \cite{CarlesonTotik}, Theorem 1.7).  Moreover, for domains like $\Omega={\mathbb C}\backslash E$, where $E$ is a compact set in ${\mathbb R}$ (e.g. Cantor-type sets), the connection between the metric properties of $E$ and the precise value of $\alpha(\Omega)$ (especially the optimal case $\alpha(\Omega)=1/2$) was studied in detail in \cite{CarlesonTotik} and \cite{Totik}. In the appendix of this paper, we will provide more examples of higher-dimensional domains with positive hyperconvexity indices.  Probably, the Teichm\"uller space of a compact Riemann surface with genus $\ge 2$ which is boundedly embedded in ${\mathbb C}^{3g-3}$ has a positive hyperconvexity index.

 For a domain $\Omega\subset{\mathbb C}^n$, let $\varrho$ be the\/ {\it relative extremal function\/} of a (fixed) closed ball $\overline{B}\subset \Omega$, i.e.,
 $$
 \varrho(z):=\varrho_{\overline{B}}(z):=\sup\{u(z):u\in PSH^-(\Omega),\,u|_{\overline{B}}\le -1\},
 $$
 where $PSH^-(\Omega)$ denotes the set of negative psh functions on $\Omega$. It is known that $\varrho$ is continuous on $\overline{\Omega}$ if $\Omega$ is a bounded hyperconvex domain (cf. \cite{BlockiLecture}, Proposition 3.1.3/vii)). Furthermore, it is easy to show that if $\alpha(\Omega)>0$ then for every $0<\alpha<\alpha(\Omega)$ there exists a constant $C>0$ such that $-\varrho\le C \delta^\alpha$.

  The goal of this paper is to present some off-diagonal estimates of the Bergman kernel on domains with positive hyperconvexity indices, in terms of $\varrho$. Usually, off-diagonal behavior of the Bergman kernel is more sensitive about the geometry of a domain than on-diagonal behavior (compare \cite{BarrettWorm}).

Let $K_\Omega(z,w)$ be the Bergman kernel of $\Omega$. It is well-known that  $K_\Omega(\cdot,w)\in L^2(\Omega)$ for all $w\in \Omega$. Thus it is natural to ask the following

\begin{problem}
 For which\/ $\Omega$ and $p> 2$ does one have $K_\Omega(\cdot,w)\in L^p(\Omega)$ for all $w\in \Omega$?
\end{problem}

For the sake of convenience, we set
$$
\beta(\Omega)=\sup\left\{\beta\ge 2: K_\Omega(\cdot,w)\in L^\beta(\Omega),\, \forall\,w\in \Omega\right\}.
$$
We call it the\/ {\it integrability index\/} of the Bergman kernel. From the well-known works of Kerzman, Catlin and Bell, we know that $\beta(\Omega)=\infty$ if $\Omega$ is a bounded pseudoconvex domain of finite D'Angelo type. On the other hand, it is not difficult to see from Barrett's work \cite{BarrettWorm} that there exist unbounded Diederich-Fornaess worm domains with $\beta(\Omega)$ arbitrarily close to $2$ (see e.g. \cite{KrantzWorm}, Lemma 7.5). Thus it is meaningful to show the following

\begin{theorem}\label{th:Main}
 If\/ $\Omega\subset {\mathbb C}^n$ is pseudoconvex, then  $\beta(\Omega)\ge 2+\frac{2\alpha(\Omega)}{2n-\alpha(\Omega)}$. Furthermore, if\/ $\Omega$ is a bounded domain with $\alpha(\Omega)>0$, then for every\/ $0<\alpha<\alpha(\Omega)$ and \/$2\le p < 2+\frac{2\alpha(\Omega)}{2n-\alpha(\Omega)}$, there exists a constant $C>0$ such that
\begin{equation}\label{eq:BerezinKernel}
\int_\Omega |K_\Omega(\cdot,w)/\sqrt{K_\Omega(w)}|^{p}\le C |\mu(w)|^{-\frac{(p-2) n}\alpha},\ \ \ w\in \Omega,
\end{equation}
where $K_\Omega(w)=K_\Omega(w,w)$ and $\mu:=|\varrho|(1+|\log|\varrho||)^{-1}$.
  \end{theorem}

       The lower bound for $\beta(\Omega)$ can be improved substantially when $n=1$:

  \begin{theorem}\label{th:OneDimension}
  If\/ $\Omega$ is a domain in ${\mathbb C}$, then $\beta(\Omega)\ge 2+\frac{\alpha(\Omega)}{1-\alpha(\Omega)}$.
\end{theorem}

In particular, we obtain the known fact that if $\Omega\subsetneq {\mathbb C}$ is a simply-connected domain then $\beta(\Omega)\ge 3$.
A famous conjecture of Brennan \cite{Brennan} suggests that the bound may be improved to $\beta(\Omega)\ge 4$; an equivalent statement is that if $f:\Omega\rightarrow {\mathbb D}$ is a conformal mapping where ${\mathbb D}$ is the unit disc, then  $f'\in L^p(\Omega)$ for all $p<4$. There is an extensive study on this conjecture (see  \cite{Bertilsson}, \cite{CarlesonJones}, \cite{CarlesonMakarov}, \cite{PommerenkeBook}, etc.).

Nevertheless, Theorem \ref{th:OneDimension} is best possible in view of the following

\begin{proposition}\label{prop:Planar}
Let $E\subset {\mathbb C}$ be a compact set satisfying ${\rm Cap}(E)>0$ and ${\rm dim}_{\rm H}(E)< 1$, where  ${\rm Cap}$ and ${\rm dim}_{\rm H}$ denote the logarithmic capacity and the Hausdorff dimension respectively. Set $\Omega:={\mathbb C}\backslash E$. Then
  $\beta(\Omega)\le 2+\frac{{\rm dim}_{\rm H}(E)}{1-{\rm dim}_{\rm H}(E)}$.
\end{proposition}

\begin{example}[1.1]
There exists a Cantor-type set $E$ with ${\rm dim}_{\rm H}(E)=0$ and ${\rm Cap}(E)>0$ $($cf. \cite{Carleson}, \S\,4, Theorem 5$)$. Thus $\beta({\mathbb C}\backslash E)=2$ in view of Proposition \ref{prop:Planar}.
\end{example}

\begin{example}[1.2]
Andrievskii \cite{Andrievskii} constructed a compact set $E\subset {\mathbb R}$ with ${\rm dim}_{\rm H}(E)=1/2$ and $\alpha({\mathbb C}\backslash E)= 1/2$. It follows from Theorem \ref{th:OneDimension} and Proposition \ref{prop:Planar} that $\beta({\mathbb C}\backslash E)=3$.
\end{example}

\begin{problem}
Is there a bounded domain $\Omega\subset {\mathbb C}$ with $\beta(\Omega)=2$?
\end{problem}

The above theorems shed some light on the study of the Bergman space
$$
A^p(\Omega)=\left\{f\in {\mathcal O}(\Omega):\int_\Omega |f|^p<\infty\right\}
$$
for domains with positive hyperconvexity indices.
For instance, we can show that
   $A^p(\Omega)\cap A^2(\Omega)$ lies dense in $A^2(\Omega)$ for suitable $p>2$ and the reproducing property of $K_\Omega(z,w)$ holds in $A^p(\Omega)$ for suitable $p<2$ (see \S\,4). A related problem is to study whether the Bergman projection can be extended to a bounded projection $L^p(\Omega)\rightarrow A^p(\Omega)$ for all $p$ in some nonempty open interval around $2$. For flat Hartogs triangles, a complete answer was recently given by Edholm-McNeal \cite{EdholmMcNeal}. For more information on this matter, we refer the reader to Lanzani's review article \cite{Lanzani} and the references therein.

   Set
 $$
 K_{\Omega,p}(z):=\sup\{|f(z)|: f\in A^p(\Omega),\|f\|_{L^p(\Omega)}\le 1\}.
 $$
Using $f:=\frac{K_\Omega(\cdot,z)}{\sqrt{K_\Omega(z)}}/\|\frac{K_\Omega(\cdot,z)}{\sqrt{K_\Omega(z)}}\|_{L^p(\Omega)}$ as a candidate, we conclude from estimate (\ref{eq:BerezinKernel}) that

   \begin{corollary}\label{coro:LeviProblem}
   Let\/ $\Omega\subset {\mathbb C}^n$ be a bounded domain with $\alpha(\Omega)>0$. For every  $p<2+\frac{2\alpha(\Omega)}{2n-\alpha(\Omega)}$, one has
   $$
    K_{\Omega,p}(z) \ge C_{\alpha,p}\, \sqrt{K_\Omega(z)}\,|\mu(z)|^{\frac{(p-2)n}{p\alpha}}.
   $$
     \end{corollary}

     \begin{remark}
      If\/ $\Omega$ is a bounded pseudoconvex domain with $C^2-$boundary, then $K_\Omega(z)\ge C\delta(z)^{-2}$ in view of the Ohsawa-Takegoshi extension theorem \cite{OhsawaTakegoshi87}. On the other hand, Hopf's lemma implies $|\varrho|\ge C\delta$. Thus
      $$
    K_{\Omega,p}(z) \ge C_{\alpha,p}\, \delta(z)^{-(1-\frac{(p-2)n}{p\alpha})}|\log \delta(z)|^{-\frac{(p-2)n}{p\alpha}}
   $$
   as $z\rightarrow \partial \Omega$.
   Notice also that
 $
 \frac{(p-2)n}{p\alpha}< \frac12
 \iff p< 2+\frac{2\alpha(\Omega)}{2n-\alpha(\Omega)}$.
     \end{remark}

     We would like to mention an interesting connection between Problem 1 and regularity problem of biholomorphic maps. The starting point is the following result of Lempert:

     \begin{theorem}[cf. \cite{Lempert}, Theorem 6.2]\label{th:Lempert}
      Let\/ $\Omega_1\subset {\mathbb C}^n$ be a bounded domain with $C^2-$boundary such that its Bergman projection $P_{\Omega_1}$ maps $C^\infty_0(\Omega_1)$ into $L^p(\Omega_1)$ for some $p>2$. Let\/ $\Omega_2\subset {\mathbb C}^n$ be a bounded domain with real-analytic boundary. Then any biholomorphic map $F:\Omega_1\rightarrow \Omega_2$ extends to a H\"older continuous map\/ $\overline{\Omega}_1\rightarrow \overline{\Omega}_2$.
     \end{theorem}

     Notice that if\/ $\Omega$ is a domain with $\int_\Omega |K_\Omega(\cdot,w)|^p$ locally uniformly bounded in $w$ for some $p\ge 1$, then for any $\phi\in C_0^\infty(\Omega)$,
$$
|P_\Omega (\phi)(z)|^p\le \int_{\zeta\in {\rm supp\,}\phi} |K_\Omega(\zeta,z)|^p\, \|\phi\|_{L^q(\Omega)}^p,\ \ \ (1/p+1/q=1),
$$
 so that
\begin{equation}\label{eq:BergmanProjection}
\int_{z\in \Omega}|P_\Omega (\phi)(z)|^p\le \|\phi\|_{L^q(\Omega)}^p\,\int_{\zeta\in {\rm supp\,}\phi} \int_{z\in \Omega}|K_\Omega(z,\zeta)|^p<\infty,
\end{equation}
i.e., $P_\Omega$ maps $C^\infty_0(\Omega)$ into $L^p(\Omega)$. Thus we have

  \begin{corollary}\label{coro:biholomorphic}
      Let\/ $\Omega_1\subset {\mathbb C}^n$ be a bounded domain with $C^2-$boundary such that the integral $\int_\Omega |K_\Omega(\cdot,w)|^p$ is locally uniformly bounded in $w$ for some $p>2$.  Let\/ $\Omega_2\subset {\mathbb C}^n$ be a bounded domain with real-analytic boundary. Then any biholomorphic map $F:\Omega_1\rightarrow \Omega_2$ extends to a H\"older continuous map\/ $\overline{\Omega}_1\rightarrow \overline{\Omega}_2$.
     \end{corollary}

  In particular, it follows from Corollary \ref{coro:biholomorphic} and Theorem \ref{th:Main} that any biholomorphic map between a bounded\/ {\it pseudoconvex\/} domain with $C^2-$boundary and a bounded domain with real-analytic boundary extends to a H\"older continuous map between their closures, which was first proved in Diederich-Fornaess \cite{DiederichFornaessProper}. On the other hand, Barrett \cite{BarrettIrregular} constructed a \/ {\it non-pseudoconvex\/} bounded smooth domain $\Omega\subset {\mathbb C}^2$ such that $P_\Omega$ fails to map $C^\infty_0(\Omega)$ into $L^p(\Omega)$ for any $p>2$, so that $\int_\Omega |K_\Omega(\cdot,w)|^p$ can not be locally uniformly bounded in $w$. However it is still expected that if $\Omega$ is a bounded domain with\/ {\it real-analytic\/} boundary then there exists $p>2$ such that $\int_\Omega |K_\Omega(\cdot,w)|^p$ is locally uniformly bounded in $w$.

   With the help of an elegant technique due to Blocki \cite{BlockiGreen} (see also \cite{HerbortGreen} for prior related techniques) on estimating the pluricomplex Green function, we may prove the following

  \begin{theorem}\label{th:Off-diagonal}
 Let\/ $\Omega\subset {\mathbb C}^n$ be a bounded domain with $\alpha(\Omega)>0$. For every $0<r<1$, there exists a constant $C>0$ such that
 \begin{equation}\label{eq:Off-diagonal}
 {\mathcal B}_\Omega(z,w):=\frac{|K_\Omega(z,w)|^2}{K_\Omega(z)K_\Omega(w)}\le C \left(\min\left\{\frac{\nu(z)}{\mu(w)},\frac{\nu(w)}{\mu(z)}\right\}\right)^r,\ \ \ z,w\in \Omega,
 \end{equation}
 where $\mu:=|\varrho|/(1+|\log|\varrho||)$ and $\nu:=|\varrho|(1+|\log|\varrho||)^n$.
 \end{theorem}

 We call $ {\mathcal B}_\Omega(z,w)$ the normalized Bergman kernel of $\Omega$. There is a long list of papers concerning point-wise estimates of the\/ {\it weighted\/} normalized Bergman kernel ${\mathcal B}_{\Omega,\varphi}(z,w):=\frac{|K_{\Omega,\varphi}(z,w)|^2}{K_{\Omega,\varphi}(z)K_{\Omega,\varphi}(w)}$ when $\Omega$ is ${\mathbb C}^n$ or a compact algebraic manifold, after
  a seminal paper of Christ \cite{Christ91} (see   \cite{Delin96}, \cite{Lindholm01}, \cite{MaMarinescu}, \cite{Christ13}, \cite{Zelditch16}, etc.). Quantitative measurements of positivity of $i\partial\bar{\partial}\varphi$ play a crucial role in these works.

  The basic difference between
  ${\mathcal B}_\Omega(z,w)$ and ${\mathcal B}_{\Omega,\varphi}(z,w)$ is that the former is always a\/ {\it biholomorphic invariant}.
  Skwarczy\'nski \cite{Skwarczynski} showed that
 $$
 d_S(z,w):=\left(1-\sqrt{{\mathcal B}_\Omega(z,w)}\right)^{1/2}
 $$
 gives an invariant distance on a bounded domain $\Omega$. The relationship between $d_S$ and the Bergman distance $d_B$ is as follows
 \begin{equation}\label{eq:BergmanVsSkwarczynski}
 d_B(z,w)\ge \sqrt{2}\, d_S(z,w)
 \end{equation}
 (see e.g. \cite{JarnickiPflug}, Corollary 6.4.7). By Theorem \ref{th:Off-diagonal} and (\ref{eq:BergmanVsSkwarczynski}), we may prove the following

 \begin{corollary}\label{cor:BergmanDistance}
 If\/ $\Omega$ is a bounded domain with $\alpha(\Omega)>0$, then for fixed $z_0\in \Omega$ there exists a constant $C>0$ such that
 \begin{equation}\label{eq:BlockiEstimate}
  d_B(z_0,z)\ge C \frac{|\log \delta(z)|}{\log|\log \delta(z)|}
  \end{equation}
  provided $z$  sufficiently close to $\partial \Omega$.
 \end{corollary}

 Blocki \cite{BlockiGreen} first proved (\ref{eq:BlockiEstimate}) for any bounded domain which admits a continuous negative psh function $\rho$ with $C_1\delta^{\alpha}\le -\rho\le C_2\delta^\alpha$ for some constants $C_1,C_2,\alpha>0$ (e.g. $\Omega$ is a pseudoconvex domain with Lipschitz boundary \cite{HarringtonPSH}). Diederich-Ohsawa \cite{DiederichOhsawa} proved earlier that the following weaker inequality
 $$
  d_B(z_0,z)\ge C \log|\log \delta(z)|
  $$
  holds for more general bounded domains admitting a continuous negative psh function $\rho$ with $C_1\delta^{1/\alpha}\le -\rho\le C_2\delta^\alpha$ for some constants $C_1,C_2,\alpha>0$.

 In order to study isometric imbedding of K\"ahler manifolds, Calabi \cite{Calabi} introduced the notion "diastasis".
  In \cite{Berger}, Marcel Berger wrote:
    {\it It seems to me that the notion of diastasis should make a comeback $[\cdots]$.
     For example, it would be interesting to compare the diastasis with the various types of Kobayashi metrics $($when they exist\/$)$.}

  Notice that the diastasis $D_B(z,w)$ with respect to the Bergman metric is
 $
 -\log {\mathcal B}_\Omega(z,w).
 $

  \begin{corollary}\label{cor:Comparison}
  If\/ $\Omega$ is a bounded domain with $\alpha(\Omega)>0$, then for fixed $z_0\in \Omega$ there exists a constant $C>0$ such that
  \begin{equation}\label{eq:diastasisVsKobayashi}
  D_B(z_0,z)\ge C d_K(z_0,z)
  \end{equation}
  where $d_K$ denotes the Kobayashi distance.
 \end{corollary}

 \begin{problem}
 Does one have $d_B(z_0,z)\ge C d_K(z_0,z)$ for bounded domains with $\alpha(\Omega)>0$?
 \end{problem}

 A domain $\Omega\subset {\mathbb C}^n$ is called\/ {\it weighted circular\/} if  there exists a $n-$tuple $(a_1,\cdots,a_n)$ of positive numbers such that $z\in \Omega$ implies $(e^{ia_1\theta}z_1,\cdots, e^{ia_n \theta}z_n)\in \Omega$ for any $\theta\in {\mathbb R}$.  As a final consequence of Theorem \ref{th:Off-diagonal}, we obtain

 \begin{corollary}\label{coro:biholomIneq}
  Let\/ $\Omega_1\subset {\mathbb C}^n$ be a bounded domain with $\alpha(\Omega_1)>0$. Let\/ $\Omega_2\subset {\mathbb C}^n$  be a bounded weighted circular domain
  which contains the origin.  Let\/ $0<\alpha<\alpha(\Omega_1)$ be given. Then for any biholomorphic map $F:\Omega_1\rightarrow \Omega_2$ there is a constant $C>0$ such that
  \begin{equation}\label{eq:biholomIneq}
  \delta_2 (F(z))\le C \delta_1(z)^{\frac{\alpha}{2n}},\ \ \ z\in \Omega_1.
  \end{equation}
  Here $\delta_1$ $($resp. $\delta_2)$ denotes the boundary distance of\/ $\Omega_1$ $($resp. $\Omega_2)$.
 \end{corollary}

 \begin{remark}
  Inequalities like $(\ref{eq:biholomIneq})$ are crucial in the study of regularity problem of biholomorphic maps\/ $($see e.g. \cite{DiederichFornaessProper}, \cite{Lempert}$)$.
 \end{remark}

\section{$L^2$ boundary decay estimates of the Bergman kernel}

\begin{proposition}\label{prop:BergmanIntegral}
 Let $\Omega\subset {\mathbb C}^n$ be a pseudoconvex domain. Let $\rho$ be a negative continuous psh  function on $\Omega$. Set
  $$
 \Omega_t=\{z\in \Omega:-\rho(z)>t\},\ \ \ t>0.
 $$
   Let $a>0$ be given.
  For every $0<r<1$, there exist constants $\varepsilon_r,C_r>0$ such that
  \begin{equation}\label{eq:2.1}
  \int_{-\rho\le \varepsilon} |K_\Omega(\cdot,w)|^2 \le C_{r}\, K_{\Omega_a}(w) (\varepsilon/a)^{r}
  \end{equation}
  for all $w\in \Omega_a$ and $\varepsilon\le \varepsilon_r a$.
  \end{proposition}

  The proof of the proposition is essentially same as Proposition 6.1 in \cite{BYChen}. For the sake of completeness, we include a proof here. The key ingredient is the following weighted estimate of the $L^2-$minimal solution of the $\bar{\partial}-$equation due to Berndtsson:

  \begin{theorem}[cf. \cite{BYChen}, Corollary 2.3]\label{th:Berndtsson}
Let\/ $\Omega$ be a bounded pseudoconvex domain in\/ ${\mathbb C}^n$ and $\varphi\in PSH(\Omega)$. Let $\psi$ be a continuous psh function on $\Omega$
 which satisfies $ri\partial\bar{\partial}\psi\ge i\partial\psi\wedge \bar{\partial}\psi$ as currents for some $0<r<1$. Suppose $v$ is a $\bar{\partial}-$closed $(0,1)-$form on $\Omega$ such that $\int_\Omega |v|^2 e^{-\varphi}<\infty$. Then the $L^2(\Omega,\varphi)-$minimal solution of $\bar{\partial}u=v$ satisfies
 \begin{equation}\label{eq:L2Minimal}
 \int_\Omega |u|^2 e^{-\psi-\varphi}\le \frac{1}{1-r} \int_\Omega |v|^2_{i\partial\bar{\partial}\psi} e^{-\psi-\varphi}.
 \end{equation}
\end{theorem}

Here $|v|^2_{i\partial\bar{\partial}\psi}$ should be understood as the infimum of non-negative locally bounded functions $H$ satisfying
 $
 i \bar{v}\wedge v\le H i\partial\bar{\partial}\psi
 $
as currents.

\begin{proof}[Proof of Proposition \ref{prop:BergmanIntegral}]
Assume first that $\Omega$ is bounded. Let $\kappa:{\mathbb R}\rightarrow [0,1]$ be a smooth cut-off function such that $\kappa|_{(-\infty,1]}=1$, $\kappa|_{[3/2,\infty)}=0$ and $|\kappa'|\le 2$. We then have
$$
     \int_{-\rho\le \varepsilon} |K_\Omega(\cdot,w)|^2 \le \int_\Omega \kappa(-\rho/\varepsilon) |K_\Omega(\cdot,w)|^2.
$$
 By the well-known property of the Bergman projection, we obtain
$$
  \int_\Omega \kappa(-\rho/\varepsilon) K_\Omega(\cdot,w)\cdot \overline{K_\Omega(\cdot,\zeta)}  = \kappa(-\rho(\zeta)/\varepsilon) K_\Omega(\zeta,w)-u(\zeta), \ \ \ \zeta\in\Omega,
$$
where $u$ is the $L^2(\Omega)-$minimal solution of the equation
$$
\bar{\partial} u= \bar{\partial}(\kappa(-\rho/\varepsilon) K_\Omega(\cdot,w))=:v.
$$
Since $\kappa(-\rho(w)/\varepsilon)=0$ provided $\frac32 \varepsilon\le a$ (i.e. $\varepsilon\le 2a/3$), so we have
\begin{equation}\label{eq:BergmanUpper}
 \int_{-\rho\le \varepsilon} |K_\Omega(\cdot,w)|^2\le -u(w).
\end{equation}
  Set
            $$
      \psi=- r\log (-\rho),\ \ \ 0<r<1.
      $$
      Clearly, $\psi$ is psh and satisfies $ri\partial\bar{\partial}\psi\ge i\partial \psi\wedge \bar{\partial}\psi$, so that
             $$
 i \bar{v}\wedge v\le C_0 r^{-1} |\kappa'(-\rho/\varepsilon)|^2 |K_\Omega(\cdot,w)|^2 i\partial\bar{\partial}\psi
 $$
 for some numerical constant $C_0>0$. Thus by Theorem \ref{th:Berndtsson} we obtain
       \begin{eqnarray*}
   \int_\Omega |u|^2e^{-\psi}
   & \le &  C_{r} \int_{\varepsilon\le -\rho\le \frac32\varepsilon} |K_\Omega(\cdot,w) |^2 e^{-\psi}\\
   & \le & C_{r} \varepsilon^{r}\int_{ -\rho\le \frac32\varepsilon} |K_\Omega(\cdot,w) |^2.
      \end{eqnarray*}
                   Since $e^{-\psi}\ge a^{r}$ on $\Omega_a$ and $u$ is holomorphic there, it follows that
    \begin{eqnarray*}
   |u(w)|^2 & \le & K_{\Omega_a}(w) \int_{\Omega_a}|u|^2\\
   & \le &  K_{\Omega_a}(w) a^{-r} \int_{\Omega}|u|^2 e^{-\psi}\\
   & \le & C_{r} K_{\Omega_a}(w) (\varepsilon/a)^{r}\int_{ -\rho\le \frac32\varepsilon} |K_{\Omega}(\cdot,w) |^2.
   \end{eqnarray*}
   Thus by (\ref{eq:BergmanUpper}), we obtain
   \begin{eqnarray*}
   \int_{-\rho\le \varepsilon} |K_\Omega(\cdot,w)|^2 \le C_{r}\, K_{\Omega_a}(w)^{1/2}  (\varepsilon/a)^{r/2}\left(\int_{ -\rho\le \frac32\varepsilon} |K_{\Omega}(\cdot,w) |^2 \right)^{1/2}.
   \end{eqnarray*}
   Notice that
   $$
   \int_{-\rho\le \frac32\varepsilon} |K_\Omega(\cdot,w) |^2 \le \int_\Omega |K_{\Omega}(\cdot,w) |^2 =K_{\Omega}(w)\le K_{\Omega_a}(w)
   $$
   provided $\frac32\varepsilon\le a$.
   Thus
   $$
   \int_{ -\rho\le\varepsilon} |K_\Omega(\cdot,w)|^2 \le C_{r}\, K_{\Omega_a}(w)(\varepsilon/a)^{r/2}.
   $$
   Replacing $\varepsilon$ by $\frac32\varepsilon$ in the argument above, we obtain
  \begin{eqnarray*}
   \int_{ -\rho\le \frac32\varepsilon} |K_\Omega(\cdot,w)|^2  & \le & C_r\,K_{\Omega_a}(w) (3/2)^{r/2} (\varepsilon/a)^{r/2}
   \end{eqnarray*}
   provided $(3/2)^2\varepsilon\le a$.
  Thus we may improve the upper bound by
  $$
   \int_{ -\rho\le\varepsilon} |K_\Omega(\cdot,w)|^2 \le C_r\,K_{\Omega_a}(w) (\varepsilon/a)^{r/2+r/4}.
   $$
   By induction, we conclude that for every $k\in {\mathbb Z}^+$,
  $$
   \int_{ -\rho\le\varepsilon} |K_\Omega(\cdot,w)|^2 \le C_{r,k}\, K_{\Omega_a}(w) (\varepsilon/a)^{r/2+r/4+\cdots+r/2^k}
   $$
   provided $(3/2)^k\varepsilon\le a$. Since $r/2+r/4+\cdots+r/2^k\rightarrow 1$ as $k\rightarrow \infty$ and $r\rightarrow 1$, we get the desired estimate under the assumption that $\Omega$ is bounded.

   In general, $\Omega$ may be exhausted by an increasing sequence $\{\Omega_j\}$ of bounded pseudoconvex domains. From the argument above we know that
   $$
  \int_{\Omega_j\cap \{-\rho\le \varepsilon\}} |K_{\Omega_j}(\cdot,w)|^2 \le C_{r}\, K_{\Omega_j\cap\Omega_a}(w) (\varepsilon/a)^{r}
  $$
   holds for all $j\gg 1$. Since $\Omega_j\uparrow \Omega$, it is well-known that $K_{\Omega_j}(\cdot,w)\rightarrow K_\Omega(\cdot,w)$ locally uniformly in $\Omega$ and $K_{\Omega_j\cap\Omega_a}(w)\rightarrow K_{\Omega_a}(w)$. It follows from Fatou's lemma that
   \begin{eqnarray*}
     \int_{-\rho\le \varepsilon} |K_\Omega(\cdot,w)|^2 & = & {\lim\inf}_{j\rightarrow \infty} \int_{\Omega_j\cap \{-\rho\le \varepsilon\}} |K_{\Omega_j}(\cdot,w)|^2 \\
     & \le & C_{r}\, K_{\Omega_a}(w) (\varepsilon/a)^{r}.
   \end{eqnarray*}
\end{proof}

\begin{remark}
  One of the referees kindly suggested an alternative proof as follows. Berndtsson-Charpentier \cite{BerndtssonCharpentier} showed that  if $\int_\Omega |f|^2 |\rho|^{-r}<\infty$ for some $0<r<1$, then
 $$
 \int_\Omega |P_\Omega(f)|^2 |\rho|^{-r}\le C_r \int_\Omega |f|^2 |\rho|^{-r}<\infty
 $$
 where $P_\Omega (f)(z):=\int_\Omega K_\Omega (z,\cdot) f(\cdot)$ is the Bergman projection. If one applies $f=\chi_{\Omega_a}K_{\Omega_a}(\cdot,w)$ where $\chi_{\Omega_a}$ denotes the characteristic function function on $\Omega_a$, then $K_\Omega(z,w)=P_\Omega(f)(z)$ and
 $$
 \int_\Omega |K_\Omega(\cdot,w)|^2 |\rho|^{-r} \le C_r \int_{\Omega_a} |K_{\Omega_a}(\cdot,w)|^2 |\rho|^{-r},
 $$
 from which the estimate $(\ref{eq:2.1})$ immediately follows.
\end{remark}

Let $\varrho$ be the relative extremal function of a (fixed) closed ball $\overline{B}\subset \Omega$. We have

  \begin{proposition}\label{prop:BerezinIntegral}
 Let $\Omega\subset {\mathbb C}^n$ be a bounded domain with $\alpha(\Omega)>0$.
  For every $0<r<1$, there exist constants $\varepsilon_r,C_r>0$ such that
  \begin{equation}\label{eq:BerezinBoundaryDecay}
  \int_{-\varrho\le \varepsilon} |K_\Omega(\cdot,w)|^2/K_\Omega(w)  \le C_{r}\, (\varepsilon/\mu(w))^{r}
  \end{equation}
  for all $\varepsilon\le \varepsilon_r \mu(w)$, where $\mu=|\varrho|(1+|\log|\varrho||)^{-1}$.
  \end{proposition}

  In order to prove this proposition, we need an elementary estimate of the pluricomplex Green function.   Recall  that the {\it pluricomplex Green function} $g_\Omega(z,w)$ of a domain $\Omega\subset {\mathbb C}^n$ is defined as
  $$
  g_\Omega(z,w)=\sup\left\{u(z):u\in PSH^-(\Omega), u(z)\le \log |z-w|+O(1)\ {\rm near\ }w\right\}.
    $$

    We first show the following quasi-H\"older-continuity of $\varrho$:

  \begin{lemma}\label{lm:Green_Holder2}
   Let $\Omega\subset {\mathbb C}^n$ be a bounded domain with $\alpha(\Omega)>0$.
 For every $r>1$ and $0<\alpha<\alpha(\Omega)$, there exists a constant $C>0$ such that
\begin{equation}\label{eq:GreenHolder2}
\varrho(z_2)\ge r\varrho(z_1)-C |z_1-z_2|^\alpha,\ \ \ z_1,z_2\in \Omega.
\end{equation}
\end{lemma}

\begin{proof}
Choose $\rho\in C(\Omega)\cap PSH^-(\Omega)$ with $-\rho\le C_\alpha\delta^\alpha$. Clearly we have
$$
\varrho(z)\ge \frac{\rho(z)}{\inf_{\overline{B}}|\rho|}\ge - C_\alpha\delta^\alpha.
$$
 To get (\ref{eq:GreenHolder2}), we employ a well-known technique of Walsh \cite{Walsh} as follows.
 Set $\varepsilon:=|z_1-z_2|$, ${\Omega}':=\Omega-(z_1-z_2)$ and
 $$
u(z)=\left\{
\begin{array}{ll}
 \varrho (z) & {\rm if\ } z\in \Omega\backslash {\Omega}'\\
 \max\left\{ \varrho (z), r\varrho(z+z_1-z_2)-C\varepsilon^\alpha\right\} & {\rm if\ } z\in \Omega\cap {\Omega}'.
\end{array}
\right.
$$
 We claim that $u\in PSH^-(\Omega)$ provided $C\gg 1$. Indeed, if $z\in \Omega\cap \partial {\Omega}'$ then $\delta(z)\le \varepsilon$, so that
 $$
  \varrho (z)\ge -C_\alpha\delta(z)^\alpha\ge -C_\alpha\varepsilon^\alpha\ge r\varrho(z+z_1-z_2)-C_\alpha\varepsilon^\alpha.
  $$
 Moreover, if $\varepsilon\le \varepsilon_r\ll1$ then $\varrho(z+z_1-z_2)\le -1/r$ for $z\in \overline{B}$ since $\varrho$ is continuous on $\overline{\Omega}$. Thus $u|_{\overline{B}}\le -1$. Since $z_2=z_1-(z_1-z_2)\in \Omega\cap {\Omega}'$, it follows that
  $$
  \varrho(z_2)\ge u(z_2)\ge r\varrho(z_1)-C_\alpha\varepsilon^\alpha.
  $$
  If $\varepsilon=|z_1-z_2|> \varepsilon_r$, then $(\ref{eq:GreenHolder2})$ trivially holds.
 \end{proof}

 \begin{remark}
  It is not known whether $\varrho$ is H\"older continuous on $\overline{\Omega}$. The answer is positive if $n=1$ $($see \cite{CarlesonGamelin}, p.\,138$)$.
 \end{remark}

 \begin{proposition}\label{prop:GreenLowerEstimate}
Let\/ $\Omega\subset {\mathbb C}^n$ be a bounded domain with $\alpha(\Omega)>0$. There exists a constant $C\gg 1$ such that
\begin{equation}\label{eq:GreenLowerEstimate}
\{g_\Omega(\cdot,w)<-1\}\subset \{\varrho< -C^{-1}\mu(w)\},\ \ \ w\in \Omega.
\end{equation}
\end{proposition}

\begin{proof}
  Fix $0<\alpha<\alpha(\Omega)$. We have $-\varrho\le C_\alpha \delta^\alpha$ for some constant $C_\alpha>0$. Clearly, it suffices to consider the case when $|\varrho(w)|\le 1/2$.  Applying Lemma \ref{lm:Green_Holder2} with $r=3/2$, we see that if $\varrho(z)=\varrho(w)/2$ then
 $$
 C_1|z-w|^\alpha\ge \frac32\varrho(z)-\varrho(w)=-\frac14 \varrho(w),
 $$
 so that
 $$
 \log\frac{|z-w|}{R}\ge \frac1{\alpha}\log|\varrho(w)|/(4C_1)-\log R\ge C_2\log|\varrho(w)|
 $$
 for some constant $C_2\gg 1$.
  It follows that
$$
\psi(z):=\left\{
\begin{array}{ll}
 \log\frac{|z-w|}{R} & {\rm if\ } \varrho(z)\le \varrho(w)/2\\
 \max\left\{\log\frac{|z-w|}{R},2C_2(\varrho(w)^{-1}\log|\varrho(w)|)\varrho(z)\right\} & {\rm otherwise}.
\end{array}
\right.
$$
is a well-defined negative psh function on $\Omega$ with a logarithmic pole at $w$, and if $\varrho(z)\ge \varrho(w)/2$, then
\begin{equation}\label{eq:GreenLowerBound}
g_\Omega(z,w)\ge \psi(z)\ge 2C_2(\varrho(w)^{-1}\log|\varrho(w)|)\varrho(z).
\end{equation}
Thus
$$
\{g_\Omega(\cdot,w)<-1\}\cap \{\varrho\ge \varrho(w)/2\}\subset \{\varrho< -C^{-1}\mu(w)\}
$$
provided $C\gg 1$. Since $\{\varrho< \varrho(w)/2\}\subset \{\varrho< -C^{-1}\mu(w)\}$ if $C\gg 1$, we conclude the proof.
\end{proof}

  \begin{proof}[Proof of Proposition \ref{prop:BerezinIntegral}]
  Set $A_w:=\{g_\Omega(\cdot,w)<-1\}$. It is known from \cite{HerbortHyperconvex} or \cite{Chen99} that
\begin{equation}\label{eq:BergmanCompare}
K_{A_w}(w)\le C_n K_\Omega(w).
\end{equation}
By Proposition \ref{prop:GreenLowerEstimate}, we have
\begin{equation}\label{eq:GreenLevel}
A_w\subset \Omega_{a(w)}:=\{\varrho< -a(w)\}
\end{equation}
where $a(w):=C^{-1}\mu(w)$ with $C\gg 1$.  If we choose $\rho=\varrho$ in Proposition \ref{prop:BergmanIntegral}, it follows that for every $\varepsilon\le \varepsilon_r a(w)$,
  \begin{eqnarray}\label{eq:BergmanIntegral}
  \int_{-\varrho\le \varepsilon} |K_\Omega(\cdot,w)|^2 & \le & C_{r}\, K_{\Omega_{a(w)}}(w) (\varepsilon/a(w))^{r}\nonumber\\
  & \le & C_{n,r}\, K_\Omega(w) (\varepsilon/a(w))^{r}
  \end{eqnarray}
   in view of (\ref{eq:BergmanCompare}), (\ref{eq:GreenLevel}).
   \end{proof}
 \section{$L^p-$integrability of the Bergman kernel}

\begin{proof}[Proof of Theorem \ref{th:Main}]
 Without loss of generality, we may assume $\alpha(\Omega)>0$. For every $0<\alpha<\alpha(\Omega)$, we may choose $\rho\in PSH^-(\Omega)$ such that
$$
-\rho\le C_\alpha \delta^\alpha
$$
for some constant $C_\alpha>0$. Let $S$ be a compact set in $\Omega$ and let $w\in S$. By virtue of Proposition \ref{prop:BergmanIntegral}, we conclude that for every $0<r<1$,
$$
 \int_{-\rho\le\varepsilon} |K_{\Omega}(\cdot,w)|^2 \le C \varepsilon^r
$$
 where $C=C(n,r,\alpha,S)>0$. Since $\{\delta\le \varepsilon\}\subset \{-\rho\le C_\alpha\varepsilon^\alpha\}$, it follows that
$$
\int_{ \delta \le \varepsilon}|K_{\Omega}(\cdot,w)|^2 \le C\varepsilon^{r\alpha}.
$$
 Since $|\delta(\zeta)-\delta(z)|\le |\zeta-z|$, we have $B(z,\delta(z))\subset \{\delta\le 2\delta(z)\}$. By the mean value inequality, we get
\begin{equation}\label{eq:upper}
|K_\Omega(z,w)|^2\le C_n \delta(z)^{-2n} \int_{\delta\le 2\delta(z)}|K_{\Omega}(\cdot,w)|^2 \le C \delta(z)^{r\alpha-2n}.
\end{equation}
Thus for every $\tau>0$ we have
\begin{eqnarray*}
\int_\Omega |K_{\Omega}(\cdot,w)|^{2+\tau} & = & \int_{\delta>1/2} |K_{\Omega}(\cdot,w)|^{2+\tau} +\sum_{k=1}^\infty \int_{2^{-k-1}<\delta\le 2^{-k}} |K_{\Omega}(\cdot,w)|^{2+\tau}\\
& \le & C\, 2^{n\tau}\int_\Omega |K_{\Omega}(\cdot,w)|^{2}+C\,\sum_{k=1}^\infty2^{(k+1)\tau(n-r\alpha/2)}\int_{\delta\le 2^{-k}} |K_{\Omega}(\cdot,w)|^{2}\\
& \le & C  +C\, 2^{\tau(n-r\alpha/2)} \sum_{k=1}^\infty 2^{-k(r\alpha+\tau(r\alpha/2-n))}\\
& < & \infty
\end{eqnarray*}
provided $\tau<\frac{2r\alpha}{2n-r\alpha}$. Since $r$ and $\alpha$ can be arbitrarily close to $1$ and $\alpha(\Omega)$ respectively, we conclude the proof of the first statement.

Since $\{\delta\le \varepsilon\}\subset \{-\varrho\le C_\alpha \varepsilon^\alpha\}$, it follows from Proposition \ref{prop:BerezinIntegral} that
\begin{equation}\label{eq:BerezinBoundaryDecay2}
\int_{\delta\le \varepsilon} |K_\Omega(\cdot,w)|^2/K_\Omega(w) \le C_{\alpha,r}(\varepsilon^\alpha/\mu(w))^r
\end{equation}
provided $\varepsilon^\alpha/\mu(w)\le \varepsilon_r\ll1$.
For every $z\in \Omega$, we have
\begin{equation}\label{eq:Pointwise1}
|K_\Omega(z,w)|^2/K_\Omega(w)\le K_\Omega(z)\le C_n \delta(z)^{-2n},
\end{equation}
and if $(2\delta(z))^\alpha\le \varepsilon_r \mu(w)$,
\begin{eqnarray}\label{eq:Pointwise2}
|K_\Omega(z,w)|^2 & \le & C_n \delta(z)^{-2n} \int_{\delta\le 2\delta(z)}|K_{\Omega}(\cdot,w)|^2\nonumber\\
                  & \le & C_{\alpha,r} K_\Omega(w)\mu(w)^{-r}\delta(z)^{\alpha r-2n}.
\end{eqnarray}
For every $\tau<\frac{2r\alpha}{2n-r\alpha}$, we conclude from (\ref{eq:Pointwise1}) that
\begin{eqnarray}\label{eq:Berezin1}
 && \int_{2\delta\ge (\epsilon_r \mu(w))^{1/\alpha}} |K_\Omega(\cdot,w)|^{2+\tau}\nonumber\\
  & \le & C_n K_\Omega(w)^{\tau/2} \int_{2\delta\ge (\epsilon_r \mu(w))^{1/\alpha}} |K_\Omega(\cdot,w)|^2\delta^{-n\tau}\nonumber\\
  & \le & C_{\alpha,r} \frac{K_\Omega(w)^{\tau/2}}{\mu(w)^{n\tau/\alpha}}\int_\Omega |K_\Omega(\cdot,w)|^2\nonumber\\
& \le & C_{\alpha,r} \frac{K_\Omega(w)^{1+\tau/2}}{\mu(w)^{n\tau/\alpha}}.
\end{eqnarray}
Now choose $k_w\in {\mathbb Z}^+$ such that $(\epsilon_r \mu(w))^{1/\alpha}\in (2^{-k_w-1},2^{-k_w}]$ (it suffices to consider the case when $\mu(w)$ is sufficiently small). We then have
\begin{eqnarray}\label{eq:Berezin2}
 && \int_{2\delta< (\epsilon_r \mu(w))^{1/\alpha}} |K_\Omega(\cdot,w)|^{2+\tau}\nonumber\\
  & \le & \sum_{k=k_w}^\infty \int_{2^{-k-1}<\delta\le 2^{-k}}
 |K_\Omega(\cdot,w)|^{2+\tau}\nonumber\\
& \le & C_{\alpha,r,\tau} \frac{K_\Omega(w)^{\tau/2}}{\mu(w)^{\tau r/2}}\sum_{k=k_w}^\infty  2^{k\tau(n-r\alpha/2)}\int_{\delta\le 2^{-k}} |K_{\Omega}(\cdot,w)|^{2}\ \ \ ({\rm by\ }(\ref{eq:Pointwise2}))\nonumber\\
& \le & C_{\alpha,r,\tau}\frac{K_\Omega(w)^{1+\tau/2}}{\mu(w)^{r(1+\tau /2)}}\sum_{k=k_w}^\infty   2^{-k(r\alpha+\tau(r\alpha/2-n))}  \ \ \ ({\rm by\ }(\ref{eq:BerezinBoundaryDecay2}))\nonumber\\
& \le & C_{\alpha,r,\tau}\frac{K_\Omega(w)^{1+\tau/2}}{\mu(w)^{r(1+\tau /2)}}\mu(w)^{(r\alpha+\tau(r\alpha/2-n))/\alpha}\nonumber\\
& \le & C_{\alpha,r,\tau} \frac{K_\Omega(w)^{1+\tau/2}}{\mu(w)^{\tau n/\alpha}}.
\end{eqnarray}
By (\ref{eq:Berezin1}) and (\ref{eq:Berezin2}), (\ref{eq:BerezinKernel}) immediately follows.
\end{proof}

 \begin{proof}[Proof of Theorem \ref{th:OneDimension}]
It suffices to use the following lemma instead of (\ref{eq:upper}) in the proof of the first statement in Theorem \ref{th:Main}.
\end{proof}

\begin{lemma}\label{lm:upper_bound}
 Let $\Omega$ be a domain in ${\mathbb C}$. For every compact set $S\subset \Omega$ and $\alpha<\alpha(\Omega)$, there exists a constant $C>0$ such that
 $$
 |K_\Omega(z,w)|\le C\delta(z)^{\alpha-1},\ \ \ z\in \Omega,\,w\in S.
 $$
 \end{lemma}

 \begin{proof}
  Let $g_\Omega(z,w)$ be the (negative) Green function on $\Omega$. Let $\Delta(c,r)$ be the disc with centre $c$ and radius $r$. Fix $w\in S$ and $z\in \Omega$ for a moment. Clearly, it suffices to consider the case when $\delta(z)\le \delta(w)/4$. Since $g_\Omega(\xi,\zeta)$ is harmonic in $\xi\in \Delta(z,\delta(z))$ and $\zeta\in \Delta(w,\delta(w)/2)$ respectively, we conclude from Poisson's formula that
  \begin{eqnarray*}
  g_\Omega(\xi,\zeta) & = & \frac1{4\pi^2}\int_0^{2\pi}\int_0^{2\pi}g_\Omega\left(z+\frac{\delta(z)}2e^{i\theta},w+\frac{\delta(w)}2e^{i\vartheta}\right)\\
  && \ \ \ \ \ \ \ \ \ \ \ \ \ \ \ \frac{\frac{\delta(z)^2}4-|\xi-z|^2}{\left|\frac{\delta(z)}2e^{i\theta}-(\xi-z)\right|^2} \frac{\frac{\delta(w)^2}4-|\zeta-w|^2}{\left|\frac{\delta(w)}2e^{i\vartheta}-(\zeta-w)\right|^2}d\theta d\vartheta
  \end{eqnarray*}
  where $\xi\in \Delta(z,\delta(z)/4)$ and $\zeta\in \Delta(w,\delta(w)/4)$. By the extremal property of $g_\Omega$, it is easy to verify
  $
  -g_\Omega\le C\delta(z)^\alpha
  $
  on $\partial \Delta(z,\delta(z)/2)\times \partial \Delta(w,\delta(w)/2)$.
   Thus
  $$
  \left| \frac{\partial^2 g_\Omega(\xi,\zeta)}{\partial\xi\partial\bar{\zeta}}\right|\le C\delta(z)^{\alpha-1}.
  $$
  Together with Schiffer's formula
  $
  K_\Omega(\xi,\zeta)=\frac2{\pi}  \frac{\partial^2 g_\Omega(\xi,\zeta)}{\partial\xi\partial\bar{\zeta}}
  $
  (cf. \cite{Schiffer}), the assertion immediately follows.
 \end{proof}

 In order to prove Proposition \ref{prop:Planar}, we need the following

 \begin{theorem}[cf. \cite{Carleson}, \S\,6, Theorem 1]\label{th:Carleson}
 Let $\Omega={\mathbb C}\backslash E$ where $E\subset {\mathbb C}$ is a compact set. Then
   \begin{enumerate}
    \item $A^2(\Omega)\neq \{0\}$ if and only if ${\rm Cap}(E)>0$.
    \item $A^p(\Omega)=\{0\}$ if $\Lambda_{2-q}(E)<\infty$, $2<p< \infty$, $\frac1p+\frac1q=1$. Here $\Lambda_s(E)$ denotes the $s-$dimensional Hausdorff measure of $E$.
   \end{enumerate}
   \end{theorem}

   \begin{remark}
    Let $\Omega\subset {\mathbb C}$ be a  domain and $E$ a closed polar set in $\Omega$. It is well-known that $E$ is removable for negative harmonic functions, so that $g_{\Omega\backslash E}(z,w)=g_\Omega(z,w)$ for $z,w\in \Omega\backslash E$. Thus $K_{\Omega\backslash E}(z,w)= K_{\Omega}(z,w)$ in view of Schiffer's formula. By the reproducing property of the Bergman kernel, we immediately get the known fact that $A^2(\Omega\backslash E)=A^2(\Omega)$.
   \end{remark}

     \begin{proof}[Proof Proposition \ref{prop:Planar}]
   Suppose on the contrary $\beta(\Omega)> 2+\frac{{\rm dim}_{\rm H}(E)}{1-{\rm dim}_{\rm H}(E)}$. Fix
    $$
    \beta(\Omega)>p> 2+\frac{{\rm dim}_{\rm H}(E)}{1-{\rm dim}_{\rm H}(E)}
    $$
    and let $q$ be the conjugate exponent of $p$, i.e., $\frac1p+\frac1q=1$.
    We then have $K_\Omega(\cdot,w)\in A^p(\Omega)$ for fixed $w$. Since
    $$
    {\rm dim}_{\rm H}(E)=\sup\{s:\Lambda_s(E)=\infty\}
    $$
    and $2-q>{\rm dim}_{\rm H}(E)$, it follows that $\Lambda_{2-q}(E)<\infty$, so that $K_\Omega(\cdot,w)=0$ in view of Theorem \ref{th:Carleson}/(2). On the other side, since  ${\rm Cap}(E)>0$, so $K_\Omega(\cdot,w)\neq 0$ in view of Theorem \ref{th:Carleson}/(1), which is absurd.
   \end{proof}

   Theorem \ref{th:OneDimension} implies $\beta(\Omega)\rightarrow \infty$ as $\alpha(\Omega)\rightarrow 1$ for planar domains (notice that $\alpha(\Omega)=1$ when $\Omega\subset {\mathbb C}$ is convex or\/ $\partial \Omega$ is $C^1$). It is also known that $\beta(\Omega)=\infty$ if $\Omega$ is a bounded smooth convex domain in ${\mathbb C}^n$ (cf. \cite{BoasStraube}). Thus it is reasonable to make the following

\begin{conjecture}
If\/ $\Omega\subset {\mathbb C}^n$ is convex, then $\beta(\Omega)=\infty$.
\end{conjecture}

\section{Applications of $L^p-$integrability of the Bergman kernel}

We first study density of $A^p(\Omega)\cap A^2(\Omega)$ in $A^2(\Omega)$.

\begin{proposition}\label{prop:density}
Let $\Omega$ be a pseudoconvex domain in ${\mathbb C}^n$. For every $1\le p<2+\frac{2\alpha(\Omega)}{2n-\alpha(\Omega)}$, $A^p(\Omega)\cap A^2(\Omega)$ lies dense in $A^2(\Omega)$.
\end{proposition}

 \begin{proof}
 Choose a sequence of functions $\chi_j\in C_0^\infty(\Omega)$ such that $0\le \chi_j\le 1$ and the sequence of sets $\{\chi_j=1\}$ exhausts $\Omega$. Given $f\in A^2(\Omega)$, we set $f_j=P_\Omega(\chi_jf)$. Clearly, $f_j\in A^p(\Omega)\cap A^2(\Omega)$ in view of Theorem \ref{th:Main} and (\ref{eq:BergmanProjection}). Moreover,
$$
\|f_j-f\|_{L^2(\Omega)}=\|P_\Omega((\chi_j-1)f)\|_{L^2(\Omega)}\le \|(\chi_j-1)f\|_{L^2(\Omega)}\rightarrow 0.
$$
\end{proof}

Similarly, we may prove the following

\begin{proposition}\label{prop:density2}
Let\/ $\Omega$ be a domain in ${\mathbb C}$. For every $1\le p<2+\frac{\alpha(\Omega)}{1-\alpha(\Omega)}$, $A^p(\Omega)\cap A^2(\Omega)$ lies dense in $A^2(\Omega)$.
\end{proposition}

Next we study reproducing property of the Bergman kernel in $A^p(\Omega)$.

\begin{proposition}\label{prop:reproduce}
Let\/ $\Omega$ be a bounded domain in ${\mathbb C}$ with $\alpha(\Omega)>0$.
 If $p>2-\alpha(\Omega)$, then
 $
 f=P_\Omega(f)
 $
 for all $f\in A^p(\Omega)$.
  \end{proposition}

 \begin{proof}
 Suppose  $f\in A^p(\Omega)$ with $p>2-\alpha(\Omega)$.  Let $q$ be the conjugate exponent of $p$. Since $q<2+\frac{\alpha(\Omega)}{1-\alpha(\Omega)}$, so the integral $
 \int_\Omega f(\cdot)K_\Omega(z,\cdot)
 $
 is well-defined in view of Theorem \ref{th:OneDimension}.
 Clearly, it suffices to consider the case $p<2$. By Theorem 1 of Hedberg \cite{Hedberg}, we may find a sequence $f_j\in {\mathcal O}(\overline{\Omega})\subset A^2(\Omega)\subset A^p(\Omega)$ such that $\|f_j-f\|_{L^p(\Omega)}\rightarrow 0$. It follows that for every $z\in \Omega$,
 $$
 f(z)=\lim_{j\rightarrow \infty} f_j(z)=\lim_{j\rightarrow \infty}\int_\Omega f_j(\cdot)K_\Omega(z,\cdot)=\int_\Omega f(\cdot)K_\Omega(z,\cdot)
 $$
 since $K_\Omega(z,\cdot)\in L^q(\Omega)$.
\end{proof}

For a bounded domain $\Omega\subset {\mathbb C}^n$, the\/ {\it Berezin transform\/} $T_\Omega$ of $\Omega$ is defined as
$$
T_\Omega(f)(z)=\int_\Omega f(\cdot) \frac{|K_\Omega(\cdot,z)|^2}{K_\Omega(z)},\ \ \ z\in \Omega,\,f\in L^\infty(\Omega).
$$
Clearly, one has $f=T_\Omega(f)$ for all $f\in A^\infty(\Omega)$.

\begin{corollary}\label{cor:BerezinReproduce1}
Let\/ $\Omega$ be a bounded domain in ${\mathbb C}$ with $\alpha(\Omega)>0$.  If $p>2/\alpha(\Omega)-1$, then  $f=T_\Omega(f)$ for all $f\in A^p(\Omega)$.
\end{corollary}

\begin{proof}
Set $p'=\frac{2p}{p+1}$. It follows from H\"older's inequality that
 \begin{eqnarray*}
 \int_\Omega |f K_\Omega(\cdot,z)|^{p'} & \le & \left(\int_\Omega |f|^{\frac{p'}{2-p'}}\right)^{2-p'}\left(\int_\Omega |K_\Omega(\cdot,z)|^{\frac{p'}{p'-1}}\right)^{p'-1}\\
                                    & = & \left(\int_\Omega |f|^{p}\right)^{2-p'}\left(\int_\Omega |K_\Omega(\cdot,z)|^{\frac{p'}{p'-1}}\right)^{p'-1}\\
                                    & < & \infty,
 \end{eqnarray*}
 since $p'>2-\alpha(\Omega)$ and $\frac{p'}{p'-1}< 2+\frac{\alpha(\Omega)}{1-\alpha(\Omega)}$. Thus $h:=fK_\Omega(\cdot,z)/K_\Omega(z)\in A^{p'}(\Omega)$ for fixed $z\in \Omega$, so that $$
 f(z)=h(z)=\int_\Omega h(\cdot) K_\Omega(z,\cdot)=\int_\Omega f(\cdot)\frac{|K_\Omega(\cdot,z)|^2}{K_\Omega(z)}.
 $$
\end{proof}

For higher-dimensional cases, we can only prove the following

\begin{proposition}\label{prop:reproduce2}
Let\/ $\Omega$ be a bounded pseudoconvex domain in ${\mathbb C}^n$. Suppose there exists a negative psh exhaustion function $\rho$ on $\Omega$ such that for suitable constants $C,\alpha>0$,
$$
|\rho(z)-\rho(w)|\le C|z-w|^{\alpha}, \ \ \ z,w\in \Omega.
$$
  For every $p>\frac{4n}{2n+\alpha}$, one has $f=P_\Omega(f)$ for all $f\in A^p(\Omega)$.
   \end{proposition}

 \begin{proof}
Set $\Omega_t=\{-\rho>t\}$, $t\ge 0$, and $\rho_t:=\rho+t$. For every $z\in \Omega_t$, we choose $z^\ast\in \partial \Omega_t$ such that $|z-z^\ast|=\delta_t(z):=d(z,\partial \Omega_t)$.  We then have
   $$
   |\rho_t(z)|=|\rho_t(z)-\rho_t(z^\ast)|\le C |z-z^\ast|^\alpha=C \delta_t(z)^\alpha
   $$
  where $C$ is a  constant independent of $t$.
  By a similar argument as the proof of Theorem \ref{th:Main}, we may show that for fixed $w\in \Omega$,
    $$
 \int_{\Omega_t}|K_{\Omega_t}(\cdot,w)|^q\le C=C(q,w)<\infty
 $$
 holds uniformly in $t\ll 1$ for every $q<2+\frac{2\alpha}{2n-\alpha}$.
 Let $2> p>\frac{4n}{2n+\alpha}$ and $f\in A^p(\Omega)$. Fix $z\in \Omega$ for a moment. For every $t\ll1$, we have $z\in \Omega_t$ and
 \begin{equation}\label{eq:reproduce}
 f(z)=\int_{\Omega_t}f(\cdot)K_{\Omega_t}(z,\cdot).
 \end{equation}
 Notice that
 \begin{eqnarray}\label{eq:reproduce2}
 && \left| \int_\Omega f(\cdot)K_\Omega(z,\cdot)-\int_{\Omega_t} f(\cdot)K_{\Omega_t}(z,\cdot)\right|\nonumber\\
 & \le & \int_{\Omega_t} |f||K_\Omega(z,\cdot)-K_{\Omega_t}(z,\cdot)|+\int_{\Omega\backslash \Omega_t}|f||K_\Omega(z,\cdot)|\nonumber\\
 & \le & \|f\|_{L^p(\Omega)}\|K_\Omega(z,\cdot)-K_{\Omega_t}(z,\cdot)\|_{L^q(\Omega_t)}+\|f\|_{L^p(\Omega\backslash \Omega_t)}\|K_\Omega(z,\cdot)\|_{L^q(\Omega)}
 \end{eqnarray}
 where $\frac1p+\frac1q=1$ (which implies $q<2+\frac{2\alpha}{2n-\alpha}$). Take $0<\gamma\ll 1$ so that $\frac{q-\gamma}{1-\gamma/2}<2+\frac{2\alpha}{2n-\alpha}$. We then have
 \begin{eqnarray*}
  && \int_{\Omega_t} |K_\Omega(z,\cdot)-K_{\Omega_t}(z,\cdot)|^q \\
  & = & \int_{\Omega_t}|K_\Omega(z,\cdot)-K_{\Omega_t}(z,\cdot)|^\gamma |K_\Omega(z,\cdot)-K_{\Omega_t}(z,\cdot)|^{q-\gamma}\\
  & \le & \left(\int_{\Omega_t}|K_\Omega(z,\cdot)-K_{\Omega_t}(z,\cdot)|^2\right)^{\gamma/2}\left(\int_{\Omega_t}
  |K_\Omega(z,\cdot)-K_{\Omega_t}(z,\cdot)|^{\frac{q-\gamma}{1-\gamma/2}}\right)^{1-\gamma/2}
 \end{eqnarray*}
 in view of H\"older's inequality.
 Since
 \begin{eqnarray*}
  && \int_{\Omega_t}|K_\Omega(z,\cdot)-K_{\Omega_t}(z,\cdot)|^2\\
   & = & \int_{\Omega_t}|K_\Omega(z,\cdot)|^2+\int_{\Omega_t}|K_{\Omega_t}(z,\cdot)|^2\\
  && -2{\rm Re} \int_{\Omega_t} K_\Omega(z,\cdot) K_{\Omega_t}(\cdot,z)\\
  & \le & K_{\Omega_t}(z)-K_\Omega(z)\\
  & \rightarrow & 0\ \ \ (t\rightarrow 0)
 \end{eqnarray*}
 and
 \begin{eqnarray*}
  && \int_{\Omega_t}
  |K_\Omega(z,\cdot)-K_{\Omega_t}(z,\cdot)|^{\frac{q-\gamma}{1-\gamma/2}}\\
   & \le & 2^{\frac{q-\gamma}{1-\gamma/2}}\left(\int_{\Omega}
  |K_{\Omega}(z,\cdot)|^{\frac{q-\gamma}{1-\gamma/2}}+\int_{\Omega_t}
  |K_{\Omega_t}(z,\cdot)|^{\frac{q-\gamma}{1-\gamma/2}}\right)\\
   & \le & C,
 \end{eqnarray*}
 it follows from (\ref{eq:reproduce}) and (\ref{eq:reproduce2}) that $f=P_\Omega(f)$.
 \end{proof}

 Similarly, we have

 \begin{corollary}
 If $p>2n/\alpha$, then $f=T_\Omega(f)$ for all $f\in A^p(\Omega)$.
 \end{corollary}

  \section{Estimate of the pluricomplex Green function}

The goal of this section is to show the following

\begin{proposition}\label{prop:GreenEstimate2}
Let\/ $\Omega\subset {\mathbb C}^n$ be a bounded domain with $\alpha(\Omega)>0$. There exists a constant $C\gg 1$ such that
\begin{equation}\label{eq:GreenEstimate2}
\{g_\Omega(\cdot,w)<-1\}\subset \{\varrho>-C\nu(w)\},\ \ \ w\in \Omega,
\end{equation}
where $\nu=|\varrho|(1+|\log|\varrho||)^n$.
\end{proposition}

We will follow the argument of Blocki \cite{BlockiGreen} with necessary modifications. The key observation is the following

\begin{lemma}[cf. \cite{BlockiGreen}]\label{lm:GreenSymmetric}
 Let $\Omega\subset {\mathbb C}^n$ be a bounded hyperconvex domain.  Suppose $\zeta,w$ are two points in $\Omega$ such that the closed balls $\overline{B}(\zeta,\varepsilon),\,\overline{B}(w,\varepsilon)\subset {\mathbb C}^n$ and $\overline{B}(\zeta,\varepsilon)\cap \overline{B}(w,\varepsilon)=\emptyset$. Then there exists $\tilde{\zeta}\in \overline{B}(\zeta,\varepsilon)$ such that
\begin{equation}\label{eq:BlockiInequality}
|g_\Omega(\tilde{\zeta},w)|^n \le n!(\log R/\varepsilon)^{n-1}|g_\Omega(w,\zeta)|
\end{equation}
where $R:={\rm diam}(\Omega)$
\end{lemma}

For the sake of completeness, we include a proof here, which relies heavily on the following fundamental results.

\begin{theorem}[cf. \cite{Demailly87}]\label{th:Demailly}
Let $\Omega$ be a bounded hyperconvex domain in ${\mathbb C}^n$.
\begin{enumerate}
\item For every $w\in \Omega$, one has $(dd^c g_\Omega(\cdot,w))^n=(2\pi)^n\delta_w$ where $\delta_w$ denotes the Dirac measure at $w$.
\item For every $\zeta\in \Omega$ and $\eta>0$, one has $\int_\Omega (dd^c \max\{g_\Omega(\cdot,\zeta),-\eta\})^n=(2\pi)^n$.
\end{enumerate}
\end{theorem}

\begin{theorem}[cf. \cite{Blocki93}, see also \cite{BlockiLecture}]\label{th:Blocki1993}
 Let $\Omega$ be a
bounded domain in ${\mathbb C}^n$. Assume that $u,v\in PSH^-\cap L^\infty(\Omega)$ are non-positive psh functions such
that $u=0$ on $\partial \Omega$. Then
\begin{equation}\label{eq:Blocki1993}
\int_\Omega |u|^n (dd^c v)^n\le n! \|v\|_\infty^{n-1} \int_\Omega
|v|(dd^c u)^n.
\end{equation}
\end{theorem}

\begin{proof}[Proof of Lemma \ref{lm:GreenSymmetric}]
Let $\eta=\log R/\varepsilon$. Since $g_\Omega(z,\zeta)\ge \log |z-\zeta|/R$, it follows that
$$
\{g_\Omega(\cdot,\zeta)=-\eta\}\subset \overline{B}(\zeta,\varepsilon).
$$
Applying first Theorem \ref{th:Blocki1993} with $u=\max\{g_\Omega(\cdot,w),-t\}$ and $v=\max\{g_\Omega(\cdot,\zeta),-\eta\}$ then letting $t\rightarrow +\infty$, we obtain
$$
\int_\Omega |g_\Omega(\cdot,w)|^n (dd^c\max\{g_\Omega(\cdot,\zeta),-\eta\})^n\le n!(2\pi)^n\eta^{n-1}|g_\Omega(w,\zeta)|
$$
in view of Theorem \ref{th:Demailly}/(1). Since  $\overline{B}(\zeta,\varepsilon)\cap \overline{B}(w,\varepsilon)=\emptyset$, it follows that $g_\Omega(\cdot,w)$ is continuous on  $\overline{B}(\zeta,\varepsilon)$, so that there exists $\tilde{\zeta}\in \overline{B}(\zeta,\varepsilon)$ such that
$$
|g_\Omega(\tilde{\zeta},w)|=\min_{\overline{B}(\zeta,\varepsilon)}|g_\Omega(\cdot,w)|.
$$
Since the measure $(dd^c\max\{g_\Omega(\cdot,\zeta),-\eta\})^n$ is supported on $\{g_\Omega(\cdot,\zeta)=-\eta\}$ with total mass $(2\pi)^n$, we immediately get (\ref{eq:BlockiInequality}).
\end{proof}

\begin{proof}[Proof of Proposition \ref{prop:GreenEstimate2}]
 Clearly, it suffices to consider the case when $w$ is sufficiently close to $\partial \Omega$. Fix $\zeta\in \Omega$ with $\varrho(\zeta)\le 2\varrho(w)$ for a moment.  Set $\varepsilon:=|\varrho(w)|^{2/\alpha}$. Since $\varepsilon \le C_\alpha^{2/\alpha} \delta(w)^2$, we see that $\overline{B}(w,\varepsilon)\subset \Omega$ provided  $\delta(w)\le \varepsilon_\alpha\ll1$.
For every $z\in \Omega$ with $\delta(z)\le \varepsilon$, we have
\begin{equation}\label{eq:Upperbound}
|\varrho(z)|\le C_\alpha \delta(z)^\alpha\le C_\alpha \varepsilon^\alpha={C}_\alpha |\varrho(w)|^2\ \ \ (\le |\varrho(w)|/2)
\end{equation}
provided $\delta(w)\le \varepsilon_\alpha\ll1$. It follows from (\ref{eq:GreenLowerBound}) and (\ref{eq:Upperbound}) that for every $\tau>0$ there exists $\varepsilon_\tau\ll \varepsilon_\alpha $ such that
\begin{equation}\label{eq:GreenModulus}
\sup_{\delta\le \varepsilon}|g_\Omega(\cdot,w)|\le \tau
\end{equation}
provided $\delta(w)\le \varepsilon_\tau$.
Since
 $$
 C_\alpha \delta(\zeta)^\alpha\ge -\varrho(\zeta)\ge -2\varrho(w)= 2\varepsilon^{\alpha/2}
 $$
 and Lemma \ref{lm:Green_Holder2} yields
 $$
 C_1 |\zeta-w|^\alpha\ge \frac32 \varrho(w)-\varrho(\zeta)\ge -\frac12 \varrho(w)=\frac12 \varepsilon^{\alpha/2},
 $$
 it follows that if $\delta(w)\le \varepsilon_\tau\ll1$ then $\overline{B}(\zeta,\varepsilon)\subset \Omega$ and
  \begin{equation}\label{eq:intersection}
  \overline{B}(\zeta,\varepsilon)\cap \overline{B}(w,\varepsilon)=\emptyset.
  \end{equation}
  By Lemma \ref{lm:GreenSymmetric}, there exists $\tilde{\zeta}\in \overline{B}(\zeta,\varepsilon)$ such that (\ref{eq:BlockiInequality}) holds.

Now set
$$
\Psi(z):=\sup\{u(z):u\in PSH^-(\Omega),\,u|_{\overline{B}(w,\varepsilon)}\le -1\}.
$$
We claim that
\begin{equation}\label{eq:GreenVsPsi}
g_\Omega(z,w)\ge \log R/\varepsilon\,\Psi(z),\ z\in \Omega\backslash B(w,\varepsilon);\ \ \  g_\Omega(z,w) \le \log \delta(w)/\varepsilon\,\Psi(z),\ z\in \Omega.
\end{equation}
To see this, first notice that
\begin{equation}\label{eq:GreenTwoSides}
\log \frac{|z-w|}R\le g_\Omega(z,w)\le \log \frac{|z-w|}{\delta(w)},\ \ \ z\in \Omega.
\end{equation}
Since
$$
u(z)=\left\{
\begin{array}{ll}
 \log|z-w|/R & {\rm if\ } z\in B(w,\varepsilon)\\
 \max\left\{\log|z-w|/R,\log R/\varepsilon\,\Psi(z)\right\} & {\rm if\ } z\in \Omega\backslash B(w,\varepsilon)
\end{array}
\right.
$$
is a negative psh function on $\Omega$ with a logarithmic pole at $w$, it follows that
$$
g_\Omega(z,w)\ge \log R/\varepsilon\,\Psi(z),\ \ \ z\in \Omega\backslash B(w,\varepsilon).
$$
Since (\ref{eq:GreenTwoSides}) implies $g_\Omega(\cdot,w)|_{\overline{B}(w,\varepsilon)}\le \log \varepsilon/\delta(w)$, so we have
$$
\Psi(z)\ge \frac{g_\Omega(z,w)}{\log \delta(w)/\varepsilon},\ \ \ z\in \Omega.
$$
By (\ref{eq:GreenModulus}) and (\ref{eq:GreenVsPsi}) we obtain
\begin{equation}\label{eq:PsiBound}
\sup_{\delta\le \varepsilon} |\Psi|\le \frac{\tau}{\log \delta(w)/\varepsilon}.
\end{equation}
Set $\tilde{\Omega}=\Omega-(\tilde{\zeta}-\zeta)$ and
 $$
v(z)=\left\{
\begin{array}{ll}
 \Psi(z) & {\rm if\ } z\in \Omega\backslash \tilde{\Omega}\\
 \max\left\{\Psi(z),\Psi(z+\tilde{\zeta}-\zeta)-\frac{\tau}{\log \delta(w)/\varepsilon}\right\} & {\rm if\ } z\in \Omega\cap \tilde{\Omega}.
\end{array}
\right.
$$
Since $\Omega\cap \partial \tilde{\Omega}\subset \{\delta\le \varepsilon\}$, it follows from (\ref{eq:PsiBound}) that $v\in PSH^-(\Omega)$.
Since
$$
\Psi(z)\le \frac{\log|z-w|/\delta(w)}{\log R/\varepsilon},\ \ \ z\in \Omega\backslash B(w,\varepsilon)
$$
in view of (\ref{eq:GreenTwoSides}) and (\ref{eq:GreenVsPsi}), and $z+\tilde{\zeta}-\zeta\in \overline{B}(w,2\varepsilon)$ if $z\in \overline{B}(w,\varepsilon)$,
it follows from the maximal principle that
$$
v|_{\overline{B}(w,\varepsilon)}\le -\frac{\log \delta(w)/(2\varepsilon)}{\log R/\varepsilon}.
$$
Thus
$$
\Psi(\tilde{\zeta})-\frac{\tau}{\log \delta(w)/\varepsilon}\le v(\zeta)\le \frac{\log \delta(w)/(2\varepsilon)}{\log R/\varepsilon}\Psi(\zeta).
$$
Combining with (\ref{eq:intersection}) and (\ref{eq:GreenVsPsi}), we obtain
$$
g_\Omega(\zeta,w)\ge \frac{(\log R/\varepsilon)^2}{\log \delta(w)/\varepsilon\cdot\log \delta(w)/(2\varepsilon)}\left(g_\Omega(\tilde{\zeta},w)-\tau\right)\ge C_3\left(g_\Omega(\tilde{\zeta},w)-\tau\right)
$$
since $\delta(w)\ge |\varrho(w)/C_\alpha|^{1/\alpha}=\sqrt{\varepsilon}/C_\alpha^{1/\alpha}$. If we choose $\tau=\frac1{2C_3}$, then
\begin{eqnarray*}
g_\Omega(\zeta,w) & \ge &   -C_3 (n!)^{1/n} (\log R/\varepsilon)^{1-1/n} |g_\Omega(w,\zeta)|^{1/n}-1/2\ \ \ ({\rm by\ }(\ref{eq:BlockiInequality}))\\
 & \ge & -C_4 |\log|\varrho(w)||^{1-1/n} \frac{|\varrho(w)\log |\varrho(\zeta)||^{1/n}}{|\varrho(\zeta)|^{1/n}}-1/2\ \ \ ({\rm by\ }(\ref{eq:GreenLowerBound}))\\
  & \ge & -C_5  \frac{|\varrho(w)|^{1/n}|\log |\varrho(w)||}{|\varrho(\zeta)|^{1/n}}-1/2
\end{eqnarray*}
since $\varrho(\zeta)\le 2\varrho(w)$. Thus
$$
\{g_\Omega(\cdot,w)<-1\}\cap \{\varrho\le 2\varrho(w)\}\subset \{\varrho> -C\nu(w) \}
$$
provided  $C\gg 1$. Since $\{\varrho>2 \varrho(w)\}\subset  \{\varrho> -C\nu(w) \}$ if $C\gg 1$, we conclude the proof.
\end{proof}

\section{Point-wise estimate of the normalized Bergman kernel and applications}

\begin{proof}[Proof of Theorem \ref{th:Off-diagonal}]
By Proposition \ref{prop:BerezinIntegral}, we  know that for every $0<r<1$ there exist constants $\varepsilon_r,C_r>0$ such that
  $$
  \int_{-\varrho\le \varepsilon} |K_\Omega(\cdot,w)|^2/K_\Omega(w)  \le C_{r}\, (\varepsilon/\mu(w))^{r}
 $$
  for all $\varepsilon\le \varepsilon_r \mu(w)$.
 Fix $z\in \Omega$ with $b(z):=C \nu(z)\le \varepsilon_r \mu(w)$ for a moment, where $C$ is the constant in (\ref{eq:GreenEstimate2}). Let $\chi:{\mathbb R}\rightarrow [0,1]$ be a smooth function satisfying $\chi|_{(0,\infty)}=0$ and $\chi|_{(-\infty,-\log 2)}=1$. We proceed the proof in a similar way as \cite{Chen99}. Notice that $g_\Omega(\cdot,z)$ is a continuous negative psh function on $\Omega\backslash \{z\}$ which satisfies
     $$
     -i\partial\bar{\partial} \log (-g_\Omega(\cdot,z))\ge i\partial \log (-g_\Omega(\cdot,z))\wedge \bar{\partial} \log (-g_\Omega(\cdot,z))
     $$
     as currents.  By virtue of the Donnelly-Fefferman estimate (cf. \cite{DonnellyFefferman}, see also \cite{BerndtssonCharpentier}), there exists a solution of the equation
   $$
 \bar{\partial} u=K_\Omega(\cdot,w)\bar{\partial}\chi(-\log(-g_\Omega(\cdot,z)))
 $$
 such that
 \begin{eqnarray*}
 \int_\Omega |u|^2 e^{-2n g_\Omega(\cdot,z)} & \le & C_0 \int_\Omega |K_\Omega(\cdot,w)|^2 |\bar{\partial}\chi(-\log(-g_\Omega(\cdot,z)))|^2_{-i\partial\bar{\partial}\log(-g_\Omega(\cdot,z))}
 e^{-2ng_\Omega(\cdot,z)}\\
 & \le & C_n \int_{\varrho>- b(z)} |K_\Omega(\cdot,w)|^2\ \ \ ({\rm by\ }(\ref{eq:GreenEstimate2}))\\
 & \le & C_{n,r} K_\Omega(w)(\nu(z)/\mu(w))^r.
 \end{eqnarray*}
 Set
 $$
 f:=K_\Omega(\cdot,w)\chi(-\log(-g_\Omega(\cdot,z)))-u.
 $$
 Clearly, we have $f\in {\mathcal O}(\Omega)$. Since $g_\Omega(\zeta,z)=\log |\zeta-z|+O(1)$ as $\zeta\rightarrow z$ and $u$ is holomorphic in a neighborhood of $z$, it follows that $u(z)=0$, i.e. $f(z)=K_\Omega(z,w)$. Moreover, we have
 \begin{eqnarray*}
 \int_\Omega |f|^2 & \le & 2\int_{\varrho>- b(z)} |K_\Omega(\cdot,w)|^2+2\int_\Omega |u|^2\\
 & \le & C_{n,r} K_\Omega(w)(\nu(z)/\mu(w))^r
 \end{eqnarray*}
 since $g_\Omega(\cdot,z)<0$.
 Thus we get
 $$
 K_\Omega(z)\ge \frac{|f(z)|^2}{\|f\|^2_{L^2(\Omega)}}\ge C_{n,r}^{-1}\frac{|K_\Omega(z,w)|^2}{K_\Omega(w)}(\mu(w)/\nu(z))^r,
 $$
and
  $$
  {\mathcal B}_\Omega(z,w)\le C_{n,r} (\nu(z)/\mu(w))^r.
  $$
    If $b(z)> \varepsilon_r \mu(w)$, then the inequality above trivially holds since $\frac{|K_\Omega(z,w)|^2}{K_\Omega(z)K_\Omega(w)}\le 1$. By symmetry of ${\mathcal B}_\Omega$, the assertion immediately follows.
 \end{proof}

  \begin{remark}
  It would be interesting to get point-wise estimates for $\frac{|S_\Omega(z,w)|^2}{S_\Omega(z)S_\Omega(w)}$ where $S_\Omega$ is the Szeg\"o kernel $($compare \cite{ChenFu11}$)$.
 \end{remark}

 \begin{proof}[Proof of Corollary \ref{cor:BergmanDistance}]
  Let $z\in \Omega$ be an arbitrarily fixed point which is sufficiently close to $\partial \Omega$. By the Hopf-Rinow theorem, there exists a Bergman geodesic $\gamma$ jointing $z_0$ to $z$, for $ds^2_B$ is complete on $\Omega$. We may choose a finite number of points $\{z_k\}_{k=1}^m\subset \gamma$ with the following order
  $$
  z_0\rightarrow z_1\rightarrow z_2\rightarrow \cdots \rightarrow z_m \rightarrow z,
  $$
  where
  $$
  |\varrho(z_{k+1})|(1+|\log|\varrho(z_{k+1})||)^{n+2}=|\varrho(z_k)|
  $$
  and
  $$
  |\varrho(z)|(1+|\log|\varrho(z)||)^{n+2} \ge |\varrho(z_m)|.
  $$
  Since
  \begin{eqnarray*}
   \frac{\nu(z_{k+1})}{\mu(z_k)} & = & \frac{|\varrho(z_{k+1})|}{|\varrho(z_k)|}(1+|\log |\varrho(z_{k+1})||)^n (1+|\log |\varrho(z_{k})||)\\
   & \le & \frac{|\varrho(z_{k+1})|}{|\varrho(z_k)|}(1+|\log |\varrho(z_{k+1})||)^{n+1}\\
   & = &  (1+|\log |\varrho(z_{k+1})||)^{-1},
  \end{eqnarray*}
 it follows from Theorem \ref{th:Off-diagonal} that there exists $k_0\in {\mathbb Z}^+$ such that  ${\mathcal B}_\Omega(z_k,z_{k+1})\le 1/4$ for all $k\ge k_0$. By (\ref{eq:BergmanVsSkwarczynski}), we get
  $$
  d_B(z_k,z_{k+1})\ge 1.
  $$
 Notice that
  \begin{eqnarray*}
  |\varrho(z_{k_0})| & = & |\varrho(z_{k_0+1})||\log|\varrho(z_{k_0+1})||^{n+2}\\
  & \le & |\varrho(z_{k_0+2})||\log|\varrho(z_{k_0+2})||^{2(n+2)}\\
  & \le & \cdots \ \  \le  |\varrho(z_m)||\log|\varrho(z_m)||^{(m-k_0)(n+2)}.
  \end{eqnarray*}
    Thus we have
  $$
  m-k_0\ge {\rm const.}\frac{|\log|\varrho(z_m)||}{\log|\log |\varrho(z_m)||}\ge {\rm const.}\frac{|\log|\varrho(z)||}{\log|\log |\varrho(z)||},
  $$
  so that
  \begin{eqnarray*}
  d_B(z,z_0) & \ge & \sum_{k=k_0}^{m-1} d_B(z_k,z_{k+1})\ge m-k_0-1\\
   & \ge &   {\rm const.}\frac{|\log|\varrho(z)||}{|\log|\log |\varrho(z)||}\\
   &\ge & {\rm const.}\frac{|\log \delta(z)|}{\log|\log\delta(z)|},
  \end{eqnarray*}
  since $|\varrho(z)|\le C_\alpha\delta^\alpha$ for any $\alpha<\alpha(\Omega)$.
 \end{proof}

  \begin{proof}[Proof of Proposition \ref{cor:Comparison}]
  For every $0<\alpha<\alpha(\Omega)$, we have $-\varrho\le C_\alpha\delta^\alpha$. Theorem \ref{th:Off-diagonal} then yields
   $$
  D_B(z_0,z)\ge \alpha|\log\delta(z)|
  $$
  as $z\rightarrow \partial \Omega$.
  Thus it suffices to show
  \begin{equation}\label{eq:Kobayashi}
  d_K(z,z_0)\le C |\log \delta(z)|
  \end{equation}
   as $z\rightarrow \partial \Omega$. To see this, let $F_{K}$ be the Kobayashi-Royden metric. Since $F_K$ is decreasing under holomorphic mappings, we conclude that
   $
   F_K(z;X)
   $
   is dominated by the KR metric of the ball $B(z,\delta(z))$. Thus $F_K(z;X)\le C |X|/\delta(z)$, from which (\ref{eq:Kobayashi}) immediately follows (compare the proof of Proposition 7.3 in \cite{BYChen}).
 \end{proof}

 In order to prove Corollary \ref{coro:biholomIneq}, we need the following elementary fact:

 \begin{lemma}\label{lm:WeightedCircular}
If\/ $\Omega\subset {\mathbb C}^n$ is a bounded weighted circular domain which contains the origin, then $K_\Omega(z,0)=K_\Omega(0)$ for any $z\in \Omega$.
\end{lemma}

\begin{proof}
For fixed $\theta\in {\mathbb R}$ we set $F_\theta(z):=(e^{ia_1\theta}z_1,\cdots, e^{ia_n \theta}z_n)$. By the transform formula of the Bergman kernel, we have
$$
K_\Omega(F_\theta(z),0)=K_\Omega(z,0),\ \ \ z\in \Omega.
$$
It follows that for any $n-$tuple $(m_1,\cdots,m_n)$ of non-negative integers,
$$
e^{i(a_1 m_1+\cdots+a_n m_n)\theta}\,\left.\frac{\partial^{m_1+\cdots+m_n}K_\Omega(z,0)}{\partial z_1^{m_1}\cdots \partial z_n^{m_n}}\right|_{z=0}= \left.\frac{\partial^{m_1+\cdots+m_n}K_\Omega(z,0)}{\partial z_1^{m_1}\cdots \partial z_n^{m_n}}\right|_{z=0},\ \ \ \forall\,\theta\in {\mathbb R},
$$
so that $\left.\frac{\partial^{m_1+\cdots+m_n}K_\Omega(z,0)}{\partial z_1^{m_1}\cdots \partial z_n^{m_n}}\right|_{z=0}=0$ if not all $m_j$ are zero. Taylor's expansion of $K_\Omega(z,0)$ at $z=0$ and the identity theorem of holomorphic functions yield $K_\Omega(z,0)=K_\Omega(0)$ for any $z\in \Omega$.
\end{proof}

  \begin{proof}[Proof of Corollary \ref{coro:biholomIneq}]
 By Lemma \ref{lm:WeightedCircular}, we have
 \begin{eqnarray*}
 {\mathcal B}_{\Omega_2}(F(z),0) & = & K_{\Omega_2}(0) K_{\Omega_2}(F(z))^{-1}
  \ge  C^{-1}\delta_2(F(z))^{2n}.
 \end{eqnarray*}
 On the other hand, Theorem \ref{th:Off-diagonal} implies
 $$
 {\mathcal B}_{\Omega_1}(z,F^{-1}(0))\le C_\alpha\,\delta_1(z)^\alpha.
 $$
 Since ${\mathcal B}_{\Omega_2}(F(z),0)={\mathcal B}_{\Omega_1}(z,F^{-1}(0))$, we conclude the proof.
 \end{proof}

 \section{Appendix: Examples of domains with positive hyperconvexity indices}

 We start with the following almost trivial fact:

 \begin{proposition}\label{prop:properHolo}
  Let\/ $\Omega_1,\Omega_2$ be two bounded domains in ${\mathbb C}^n$ such that there exists a biholomorphic map $F:\Omega_1\rightarrow \Omega_2$ which extends to a H\"older continuous map $\overline{\Omega}_1\rightarrow \overline{\Omega}_2$. If $\alpha(\Omega_2)>0$, then $\alpha(\Omega_1)>0$.
 \end{proposition}

 \begin{proof}
 Let $\delta_1$ and $\delta_2$ denote the boundary distance of $\Omega_1$ and $\Omega_2$ respectively. Choose $\rho_2\in PSH^-\cap C(\Omega_2)$ such that $-\rho_2\le C\delta_2^\alpha$ for some $C,\alpha>0$. Set $\rho_1:=\rho_2\circ F$. Clearly, $\rho_1\in PSH^-\cap C(\Omega_1)$. For fixed $z\in \Omega_1$, we choose $z^\ast\in \partial \Omega_1$ so that $|z-z^\ast|=\delta_1(z)$. Since $F(z^\ast)\in \partial \Omega_2$, it follows that
 \begin{eqnarray*}
 -\rho_1(z) & \le &  C\delta_2(F(z))^\alpha=C(\delta_2(F(z))-\delta_2(F(z^\ast)))^\alpha\\
            & \le & C |F(z)-F(z^\ast)|^\alpha\le C|z-z^\ast|^{\gamma\alpha}\\
            & \le & C\delta_1(z)^{\gamma\alpha}
 \end{eqnarray*}
 where $\gamma$ is the order of H\"older continuity of $F$ on $\overline{\Omega}_1$.
 \end{proof}

 \begin{example}
  Let $D\subset {\mathbb C}$ be a bounded Jordan domain which admits a uniformly H\"older continuous conformal map $f$ onto the unit disc $\Delta$ $($e.g. a quasidisc with a fractal boundary\,$)$. Set $F(z_1,\cdots,z_n):=(f(z_1),\cdots,f(z_n))$. Clearly, $F$ is a biholomorphic map between $D^n$ and $\Delta^n$ which extends to a H\"older continuous map between their closures. Let
  $$
  \Omega_2:=\{z\in {\mathbb C}^n:|z_1|^{a_1}+\cdots+|z_n|^{a_n}<1\}
  $$
  where $a_j>0$. Clearly, we have $\alpha(\Omega_2)>0$.
  By Proposition \ref{prop:properHolo}, we conclude that the domain\/ $\Omega_1:=F^{-1}(\Omega_2)$ satisfies $\alpha(\Omega_1)>0$. Notice that some parts of $\partial \Omega_1$ might be highly irregular.
 \end{example}

  A domain $\Omega\subset {\mathbb C}^n$ is called ${\mathbb C}-$convex if $\Omega\cap L$ is a simply-connected domain in $L$ for every affine complex line $L$. Clearly, every convex domain is ${\mathbb C}-$convex.

\begin{proposition}\label{prop:C-convex}
If\/ $\Omega\subset {\mathbb C}^n$ is a bounded ${\mathbb C}-$convex domain, then $\alpha(\Omega)\ge 1/2$.
\end{proposition}

\begin{proof}
Let $w\in \Omega$ be an arbitrarily fixed point. Let $w^\ast$ be a point on $\partial \Omega$ satisfying $\delta(w)=|w-w^\ast|$. Let $L$ be the complex line determined by $w$ and $w^\ast$. Since every ${\mathbb C}-$convex domain is linearly convex (cf. \cite{HormanderConvexity}, Theorem 4.6.8), it follows that there exists an affine complex hyperplane $H\subset {\mathbb C}^n\backslash \Omega$ with $w^\ast\in H$. Since $|w-w^\ast|=\delta(w)$, so $H$ has to be\/ {\it orthogonal\/} to $L$. Let $\pi_L$ denote the natural projection ${\mathbb C}^n\rightarrow L$. Notice that $\pi_L(\Omega)$ is a bounded simply-connected domain in $L$ in view of \cite{HormanderConvexity}, Proposition 4.6.7. By Proposition 7.3 in \cite{BYChen}, there exists a negative continuous function $\rho_L$ on $\pi_L(\Omega)$ with
$$
(\delta_L/\delta_L(z^0_L))^{2}\le -\rho_L\le (\delta_L/\delta_L(z^0_L))^{1/2}
$$
  where $\delta_L$ denotes the boundary distance of $\pi_L(\Omega)$ and $z^0_L\in \pi_L(\Omega)$ satisfies $\delta_L(z^0_L)=\sup_{\pi_L(\Omega)}\delta_L$. Fix a point $z^0\in \Omega$. We have
  $$
  \delta_L(z^0_L)\ge \delta_L(\pi_L(z^0))\ge \delta(z^0).
  $$
   Set
 $$
 \varrho_{z_0}(z)=\sup\{u(z): u\in PSH^-(\Omega),\,u(z^0)\le -1\}.
 $$
 Clearly, $\varrho_{z_0}\in PSH^-(\Omega)$. Since $\Omega\subset \pi_L^{-1}(\pi_L(\Omega))$, it follows that $\pi_L^\ast (\rho_L)\in PSH^-(\Omega)$. Since $\pi_L^\ast(\delta_L)(w)=\delta(w)$ and
 $$
 \pi_L^\ast (\rho_L)(z^0)=\rho_L(\pi_L(z^0))\le -(\delta_L(\pi_L(z^0))/\delta_L(z^0_L))^2
 $$
  so
 \begin{eqnarray*}
 \varrho_{z_0}(w) & \ge & (\delta_L(z^0_L)/\delta_L(\pi_L(z^0)))^2\pi_L^\ast (\rho_L)(w)\\
 &\ge & -(\delta_L(z^0_L)^{3/2}/\delta_L(\pi_L(z^0))^2)\delta(w)^{1/2}\\
 & \ge & -(R^{3/2}/\delta(z^0)^2) \delta(w)^{1/2}
 \end{eqnarray*}
 where $R={\rm diam\,}(\Omega)$. Thus $\alpha(\Omega)\ge 1/2$.
\end{proof}

\begin{remark}
After the first version of this paper was finished, the author was kindly informed by Nikolai Nikolov that Proposition \ref{prop:C-convex} follows also from Proposition 3/$(ii)$ of\/ \cite{Nikolov}.
\end{remark}

Complex dynamics also provides interesting examples of domains with $\alpha(\Omega)>0$. Let $q(z)=\sum_{j=0}^d a_j z^j$ be a complex polynomial of degree $d\ge 2$. Let $q^n$ denote the $n-$iterates of $q$. The attracting basin at $\infty$ of $q$ is defined by
$$
F_\infty:=\{z\in \overline{\mathbb C}:q^n(z)\rightarrow \infty\ {\rm as\ }n\rightarrow \infty\},
$$
which is a domain in $\overline{\mathbb C}$ with $q(F_\infty)=F_\infty$. The Julia set of $q$ is defined by $J:=\partial F_\infty$. It is known that $J$ is always uniformly perfect. Thus $\alpha(F_\infty)>0$.

We say that $q$ is\/ {\it hyperbolic\/} if there exist constants $C>0$ and $\gamma>1$ such that
 $$
 \inf_J |(q^n)'|\ge C\gamma^n,\ \ \ \forall\,n\ge 1.
  $$
  Consider a holomorphic family $\{q_\lambda\}$ of hyperbolic polynomials of constant degree $d\ge 2$ over the unit disc $\Delta$. Let $F_\infty^\lambda$ denote the attracting basin at $\infty$ of $q_\lambda$ and let $J_\lambda:=\partial F_\infty^\lambda$. Let $\Omega_r$ denote the total space of $F_\infty^\lambda$ over the disc $\Delta_r:=\{z\in {\mathbb C}:|z|<r\}$ where $0<r\le 1$, that is
$$
\Omega_r=\{(\lambda,w):\lambda\in \Delta_r,\,w\in F_\infty^\lambda\}.
$$

\begin{proposition}\label{prop:basin}
For every $0<r<1$, $\Omega_r$ is a bounded domain in ${\mathbb C}^2$ with $\alpha(\Omega_r)>0$.
\end{proposition}

\begin{proof}
 We first show that $\Omega_r$ is a domain. Ma$\tilde{\rm n}\acute{\rm e}$, Sad and
Sullivan \cite{SullivanMotion} showed that there exists a family of maps $\{f_\lambda\}_{\lambda\in \Delta}$ such that
\begin{enumerate}
 \item $f_\lambda:J_0\rightarrow J_\lambda$ is a homeomorphism for each $\lambda\in \Delta$;
 \item $f_0={\rm id}|_{J_0}$;
 \item $f(\lambda,z):=f_\lambda(z)$ is holomorphic on $\Delta$ for each $z\in J_0$;
 \item $q_\lambda=f_\lambda\circ q_0\circ f_\lambda^{-1}$ on $J_\lambda$, for each $\lambda\in \Delta$.
\end{enumerate}
In other words, properties $(1)\sim(3)$ say that $\{f_\lambda\}_{\lambda\in \Delta}$ gives a\/ {\it holomorphic motion\/} of $J_0$. By a result of Slodkowski \cite{Slodkowski}, $\{f_\lambda\}_{\lambda\in \Delta}$ may be extended to a holomorphic motion $\{\tilde{f}_\lambda\}_{\lambda\in \Delta}$ of $\overline{\mathbb C}$ such that

\medskip

a)  $\tilde{f}_\lambda:\overline{{\mathbb C}}\rightarrow \overline{{\mathbb C}}$ is a
quasiconformal map of dilatation $\le
\frac{1+|\lambda|}{1-|\lambda|}$, for each $\lambda\in \Delta$;

b) $\tilde{f}_\lambda:F^0_\infty\rightarrow F^\lambda_\infty$ is a homeomorphism for each $\lambda\in \Delta$;

c)  $\tilde{f}(\lambda,z):=\tilde{f}_\lambda(z)$ is jointly H\"older continuous in $(\lambda,z)$.

\medskip

It follows immediately that $\Omega_r$ is a domain in ${\mathbb C}^n$ for each $r\le 1$. Let $\delta_\lambda$ and $\delta$ denote the boundary distance of $F_\infty^\lambda$ and $\Omega_1$ respectively. We claim that for every $0<r<1$ there exists $\gamma>0$ such that
\begin{equation}\label{eq:distComparison}
\delta_\lambda(w)\le C \delta(\lambda,w)^\gamma,\ \ \ \lambda\in \Delta_r,\,w\in F_\infty^\lambda.
\end{equation}
To see this, choose $(\lambda',w_{\lambda'})$ where $w_{\lambda'}\in J_{\lambda'}$, such that
$$
\delta(\lambda,w)=\sqrt{|\lambda-\lambda'|^2+|w-w_{\lambda'}|^2}.
$$
Write $w_{\lambda'}=\tilde{f}(\lambda',z_0)$ where $z_0\in J_0$. Since $\tilde{f}(\lambda,z_0)\in J_\lambda$, it follows that
\begin{eqnarray*}
 \delta_\lambda(w) & \le & |w-\tilde{f}(\lambda,z_0)|\le |w-w_{\lambda'}|+|\tilde{f}(\lambda',z_0)-\tilde{f}(\lambda,z_0)|\\
 & \le & |w-w_{\lambda'}|+C |\lambda-\lambda'|^\gamma\\
 & \le & \delta(\lambda,w)+C\delta(\lambda,w)^\gamma\\
 & \le & C'\delta(\lambda,w)^\gamma
\end{eqnarray*}
where $\gamma$ is the order of H\"older continuity of $\tilde{f}$ on $\Omega_r$.

¡¡Recall that the Green function $g_\lambda(w):=g_{F_\infty^\lambda}(w,\infty)$ at $\infty$ of $F_\infty^\lambda$ satisfies
 \begin{equation}\label{eq:GreenDynimics}
  g_\lambda(w)=\lim_{n\rightarrow \infty} d^{-n}\log |q_\lambda^n(w)|,\ \ \ w\in F_\infty^\lambda
 \end{equation}
 where the convergence is uniform on compact subsets of $F_\infty^\lambda$ (cf. \cite{RansfordBook}, Corollary 6.5.4). Actually the proof of Corollary 6.5.4 in \cite{RansfordBook} shows that  the convergence is also uniform on compact subsets of $\Omega_1$. Since $\log |q_\lambda^n(w)|$ is psh in $(\lambda,w)$, so is $g(\lambda,w):=g_\lambda(w)$. By (\ref{eq:distComparison}) it suffices to verify that for every $0<r<1$ there are positive constants $C,\alpha$ such that $-g_\lambda(w)\le C\,\delta_\lambda(w)^{\alpha}$ for each $\lambda\in \Delta_r$ and $w\in F_\infty^\lambda$. This can be verified similarly as the proof of Theorem 3.2 in \cite{CarlesonGamelin}.
\end{proof}

\begin{conjecture}
Let $D\subset {\mathbb C}$ be a domain with $\alpha(D)>0$. Let $\{f_\lambda\}_{\lambda\in \Delta}$ be a holomorphic motion of $D$. Let
$$
\Omega_r:=\{(\lambda,w): \lambda\in \Delta_r,\,w\in f_\lambda(D)\}.
$$
One has $\alpha(\Omega_r)>0$ for each $r<1$.
 \end{conjecture}

 \bigskip

 \textbf{Acknowledgements}

 \bigskip

 It is my pleasure to thank the valuable comments from the referees, Prof. Nikolai Nikolov and Dr. Xieping Wang.

\end{document}